\documentclass[12pt,twoside]{amsart}
\usepackage[latin1]{inputenc}
\usepackage{amsmath, amsthm, amscd, amsfonts, amssymb, graphicx}
\usepackage[bookmarksnumbered, plainpages]{hyperref}

\newtheorem{thm}{Theorem}[section]

\newtheorem{lem}[thm]{Lemma}
\newtheorem{prop}[thm]{Proposition}
\newtheorem{defn}[thm]{Definition}

\numberwithin{equation}{section}

\begin{document}

\title{\bf Anomaly cancellation formulas and $E_{8}$ bundles for almost complex manifolds }
\author{Siyao Liu \hskip 0.4 true cm Yong Wang$^{*}$ }

\thanks{{\scriptsize
\hskip -0.4 true cm \textit{2010 Mathematics Subject Classification:}
58C20; 57R20; 53C80.
\newline \textit{$^{*}$Corresponding author}}}

\maketitle

\begin{abstract}
\indent In this paper, we extend the elliptic genus in \cite{L2} by the gauge group $E_{8}$ and the gauge group $E_{8}\times E_{8}.$
Then we prove that the generalized elliptic genus are the weak Jacobi forms.
Using these elliptic genus, we obtain some $SL_{2}(\mathbf{Z})$ modular forms and get some new anomaly cancellation formulas of characteristic forms for almost complex manifolds.
\end{abstract}

\vskip 0.2 true cm


\pagestyle{myheadings}
\markboth{\rightline {\scriptsize Liu}}
         {\leftline{\scriptsize Anomaly cancellation formulas and $E_{8}$ bundles for almost complex manifolds}}

\bigskip
\bigskip


\section{ Introduction }

In \cite{AW}, Alvarez-Gaum\'{e} and Witten discovered the ``miraculous cancellation" formula.
This formula reveals a beautiful relation between the top components of the Hirzebruch $\widehat{L}$-form and $\widehat{A}$-form of a $12$-dimensional smooth Riemannian manifold $M.$
Liu established higher-dimensional ``miraculous cancellation" formulas for $(8k+4)$-dimensional Riemannian manifolds by developing modular invariance properties of characteristic forms \cite{L}.
Han and Zhang established a general cancellation formula that involves a complex line bundle for $(8k+4)$-dimensional smooth Riemannian manifold in \cite{HZ2,HZ}.
For higher-dimensional smooth Riemannian manifolds the authors obtained some cancellation formulas in \cite{HH,HLZ3,HLZ2,W1}.
And in \cite{LW}, the authors generalized the Han-Liu-Zhang cancellation formulas to the $(a, b)$ type cancellation formulas.

In \cite{L2}, Li extended the elliptic genus of an almost complex manifold to a twisted version and showed that it is a weak Jacobi form.
By the Lemma in \cite{G}, Li got some $SL_{2}(\mathbf{Z})$ modular forms of weight $d-l+n.$
From this, Li gave some new general cancellation formulas.
Wang defined a generalized elliptic genus of an almost complex manifold which generalizes the elliptic genus in \cite{L2}.
In \cite{W}, Wang proved more anomaly cancellation formulas by modular forms over $SL_{2}(\mathbf{Z})$ induced by the generalized elliptic genus.
Wang and Yang twisted the Chen-Han-Zhang $SL_{2}(\mathbf{Z})$ modular form \cite{CHZ} by $E_{8}$ bundles and got $SL_{2}(\mathbf{Z})$ modular forms of weight $14$ and $10$ for $14$ and $10$ dimensional spin manifolds \cite{WY}.
In \cite{HLZ}, Han, Liu and Zhang showed that both of the Ho$\check{r}$ava-Witten anomaly factorization formula for the gauge group $E_{8}$ and the Green-Schwarz anomaly factorization formula for the gauge group $E_{8}\times E_{8}$ could be derived through modular forms of weight $14.$

In this paper, we are interested in finding some new cancellation formulas.
According to \cite{L2} and \cite{W}, we found some of the elliptic genus of an almost complex manifold.
In this article, we generalize the previously existing elliptic genus through $E_{8}$ bundles and $E_{8}\times E_{8}$  bundles and show that the elliptic genus are weak Jacobi forms.
We get some $SL_{2}(\mathbf{Z})$ modular forms of weight $2d-l+4+n$ and $2d-l+8+n$ and we conclude some new anomaly cancellation formulas for a $2d$ dimensional almost complex manifold.

A brief description of the organization of this paper is as follows.
In Section 2, we introduce the elliptic genus for the gauge group $E_{8}$ and $E_{8}\times E_{8}$ and prove the elliptic genus are weak Jacobi forms. 
Section 3 and Section 4 contain discussions of anomaly cancellation formulas for the gauge groups $E_{8}$ and $E_{8}\times E_{8},$ respectively.


\vskip 1 true cm

\section{ The generalized elliptic genus }

Firstly, we give some definitions and basic notions that will be used throughout the paper.
For the details, see \cite{A,H,HW,K,Z}.

Suppose $(M, J)$ is a $2d$ dimensional almost complex manifold with an almost complex structure $J.$
Let $E, F$ be two Hermitian vector bundles over $M,$ for any complex number $q$
\begin{align}
\wedge_{q}(E)&=\mathbf{C}|_{M}+qE+q^{2}\wedge^{2}(E)+\cdot\cdot\cdot,\\
S_{q}(E)&=\mathbf{C}|_{M}+qE+q^{2}S^{2}(E)+\cdot\cdot\cdot,
\end{align}
denote respectively the total exterior and symmetric powers of $E,$ which live in $K(M)[[q]].$
The following relations between these operations hold
\begin{align}
S_{q}(E)=\frac{1}{\wedge_{-q}(E)},~~~~ \wedge_{q}(E-F)=\frac{\wedge_{q}(E)}{\wedge_{q}(F)}.
\end{align}
Here, for any complex vector bundle $E,$
\begin{align}
\wedge_{q}(E):=\bigoplus^{\infty}_{n=0}\wedge^{n}(E),~~~~ S_{q}(E):=\bigoplus^{\infty}_{n=0}S^{n}(E),
\end{align}
denote the generating series of the exterior and symmetric powers of $E,$ respectively.

Moreover, if $\{ \omega_{i} \}, \{ \omega'_{j} \}$ are formal Chern roots for Hermitian vector bundles $E, F$ respectively, then
\begin{align}
\text{ch}(\wedge_{q}(E))=\prod_{i}(1+e^{\omega_{i}}q).
\end{align}
Then we have the following formulas for Chern character forms,
\begin{align}
\text{ch}(S_{q}(E))=\frac{1}{\prod\limits_{i}(1-e^{\omega_{i}}q)},~~~~ \text{ch}(\wedge_{q}(E-F))=\frac{\prod\limits_{i}(1+e^{\omega_{i}}q)}{\prod\limits_{j}(1+e^{\omega'_{j}}q)}.
\end{align}

Refer to \cite{C}, we recall the four Jacobi theta functions are defined as follows:
\begin{align}
&\theta(\tau, z)=2q^{\frac{1}{8}}\sin(\pi z)\prod^{\infty}_{j=1}[(1-q^{j})(1-e^{2\pi\sqrt{-1}z}q^{j})(1-e^{-2\pi\sqrt{-1}z}q^{j})];\\
&\theta_{1}(\tau, z)=2q^{\frac{1}{8}}\cos(\pi z)\prod^{\infty}_{j=1}[(1-q^{j})(1+e^{2\pi\sqrt{-1}z}q^{j})(1+e^{-2\pi\sqrt{-1}z}q^{j})];\\
&\theta_{2}(\tau, z)=\prod^{\infty}_{j=1}[(1-q^{j})(1-e^{2\pi\sqrt{-1}z}q^{j-\frac{1}{2}})(1-e^{-2\pi\sqrt{-1}z}q^{j-\frac{1}{2}})];\\
&\theta_{3}(\tau, z)=\prod^{\infty}_{j=1}[(1-q^{j})(1+e^{2\pi\sqrt{-1}z}q^{j-\frac{1}{2}})(1+e^{-2\pi\sqrt{-1}z}q^{j-\frac{1}{2}})],
\end{align}
where $q=e^{2\pi\sqrt{-1}\tau}$ with $\tau\in\mathcal{H},$ the upper half plane.

In what follows,
\begin{align}
SL_{2}(\mathbf{Z})=\Big\{ \Big( \begin{array}{cc}
a & b\\
c & d_{0}
\end{array}
\Big)\Big|a, b, c, d_{0}\in\mathbf{Z}, ad_{0}-bc=1
\Big \}
\end{align}
stands for the modular group. Write $S=\Big(\begin{array}{cc}
0 & -1\\
1 & 0
\end{array}\Big),
T=\Big(\begin{array}{cc}
1 & 1\\
0 & 1
\end{array}\Big)$
be the two generators of $SL_{2}(\mathbf{Z}).$
They act on $\mathcal{H}$ by $S\tau=-\frac{1}{\tau}, T\tau=\tau+1.$
One has the following transformation laws of theta functions under the actions of $S$ and $T:$
\begin{align}
&\theta(\tau+1, z)=e^{\frac{\pi\sqrt{-1}}{4}}\theta(\tau, z),~~~~ \theta\Big(-\frac{1}{\tau}, z\Big)=\frac{1}{\sqrt{-1}}\Big(\frac{\tau}{\sqrt{-1}}\Big)^{\frac{1}{2}}e^{\pi\sqrt{-1}\tau z^{2}}\theta(\tau, \tau z);\\
&\theta_{1}(\tau+1, z)=e^{\frac{\pi\sqrt{-1}}{4}}\theta_{1}(\tau, z),~~~~ \theta_{1}\Big(-\frac{1}{\tau}, z\Big)=\Big(\frac{\tau}{\sqrt{-1}}\Big)^{\frac{1}{2}}e^{\pi\sqrt{-1}\tau z^{2}}\theta_{2}(\tau, \tau z);\\
&\theta_{2}(\tau+1, z)=\theta_{3}(\tau, z),~~~~ \theta_{2}\Big(-\frac{1}{\tau}, z\Big)=\Big(\frac{\tau}{\sqrt{-1}}\Big)^{\frac{1}{2}}e^{\pi\sqrt{-1}\tau z^{2}}\theta_{1}(\tau, \tau z);\\
&\theta_{3}(\tau+1, z)=\theta_{2}(\tau, z),~~~~ \theta_{3}\Big(-\frac{1}{\tau}, z\Big)=\Big(\frac{\tau}{\sqrt{-1}}\Big)^{\frac{1}{2}}e^{\pi\sqrt{-1}\tau z^{2}}\theta_{3}(\tau, \tau z).
\end{align}

Let $E_{2}(\tau)$ be Eisenstein series which is a quasimodular form over $SL_{2}(\mathbf{Z}),$ satisfying
\begin{align}
E_{2}\Big(\frac{a\tau+b}{c\tau+d_{0}}\Big)=(c\tau+d_{0})^{2}E_{2}(\tau)-\frac{6\sqrt{-1}c(c\tau+d_{0})}{\pi}.
\end{align}
In particular, we have
\begin{align}
&E_{2}(\tau+1)=E_{2}(\tau),\\
&E_{2}\Big(-\frac{1}{\tau}\Big)=\tau^{2}E_{2}(\tau)-\frac{6\sqrt{-1}\tau}{\pi},
\end{align}
and
\begin{align}
E_{2}(\tau)=1-24q-72q^{2}+\cdot\cdot\cdot,
\end{align}
where the $``\cdot\cdot\cdot"$ terms are the higher degree terms, all of which have integral coefficients.

For the principal $E_{8}$ bundles $P_{\alpha}, \alpha=1,2$ consider the associated bundles
\begin{align}
\mathcal{V}_{\alpha}=\sum_{n=0}^{\infty}(P_{\alpha}\times_{\rho_{n}}V_{n})q^{n}\in K(M)[[q]].
\end{align}
Let $\mathcal{W}_{\alpha}=P_{\alpha}\times_{\rho_{1}}V_{1}, \alpha=1,2$ be the complex vector bundles associated to the adjoint representation $\rho_{1},$ $\overline{\mathcal{W}_{\alpha}}=P_{\alpha}\times_{\rho_{2}}V_{2}, \alpha=1,2$ be the complex vector bundles associated to the adjoint representation $\rho_{2}.$

In the following, we assume that $2d<16.$
Following the notation of \cite{HLZ}, we have that there are formal two forms $y_{\kappa}^{\alpha}, 1\leq \kappa\leq 8, \alpha=1,2$ such that
\begin{align}
\varphi(\tau)^{(8)}ch(\mathcal{V}_{\alpha})=\frac{1}{2}\Big(\prod_{\kappa=1}^{8}\theta_{1}(\tau, y_{\kappa}^{\alpha})+\prod_{\kappa=1}^{8}\theta_{2}(\tau, y_{\kappa}^{\alpha})+\prod_{\kappa=1}^{8}\theta_{3}(\tau, y_{\kappa}^{\alpha})\Big),
\end{align}
and
\begin{align}
\sum_{\kappa=1}^{8}(2\pi\sqrt{-1}y_{\kappa}^{\alpha})^{2}=-\frac{1}{30}c_{2}(\mathcal{W}_{\alpha}),
\end{align}
where $\varphi(\tau)=\prod_{j=1}^{\infty}(1-q^{j}),$ $c_{2}(\mathcal{W}_{\alpha})$ denotes the second Chern class of $\mathcal{W}_{\alpha}.$

\begin{lem}(\cite{G,BL})
A holomorphic function $\phi(\tau, z):\mathcal{H}\times \mathcal{C}\rightarrow\mathcal{C}$  is called a weak Jacobi form of weight $k\in\mathbf{Z}/2$ and index $t\in\mathbf{Z}/2$ if it satisfies the functional equations
\begin{align}
\phi\Big(\frac{a\tau+b}{c\tau+d_{0}}, \frac{z}{c\tau+d_{0}}\Big)=(c\tau+d_{0})^{k}\exp{(\frac{2\pi\sqrt{-1}tcz^{2}}{c\tau+d_{0}})}\phi(\tau, z), ~\Big(\begin{array}{cc}
a & b\\
c & d_{0}
\end{array}\Big)\in SL_{2}(\mathbf{Z})
\end{align}
and
\begin{align}
\phi(\tau, z+\lambda\tau+\mu)=(-1)^{2t(\lambda+\mu)}\exp{(-2\pi\sqrt{-1}t(\lambda^{2}\tau+2\lambda z))}\phi(\tau, z), ~\lambda,\mu\in\mathbf{Z},
\end{align}
where $\mathcal{C}$ is the complex plane.
\end{lem}

Let $T$ be the holomorphic tangent bundle in the sense of $J$ and $T^{*}$ is the dual of $T.$
Let $W$ denote a complex $l$ dimensional vector bundle on $M$ and $W^{*}$ is the dual of $W.$
Set the first Chern classes of $T$ and $W$ by $c_{1}(M)$ and $c_{1}(W).$
We denote by $2\pi\sqrt{-1}x_{i}~(1\leq i\leq d)$ and $2\pi\sqrt{-1}w_{j}~(1\leq j\leq l)$ respectively the formal Chern roots of $T$ and $W.$
Then the Todd form of $(M, J)$ is defined by
\begin{align}
Td(M):=\prod_{i=1}^{d}\frac{2\pi\sqrt{-1}x_{i}~}{1-\exp{(-2\pi\sqrt{-1}x_{i})}}.
\end{align}

\begin{defn}
The generalized elliptic genus of $(M, J)$ with respect to $W,$ which we denote by $Ell(M, W, \mathcal{V}_{\alpha}, \tau, z)$ is defined by
\begin{align}
Ell(M, W, \mathcal{V}_{\alpha}, \tau, z):=&\Big\{\exp{\Big(\frac{1}{24}E_{2}(\tau)\cdot\frac{1}{30}c_{2}(\mathcal{W}_{\alpha})+\frac{c_{1}(W)-c_{1}(M)}{2}\Big)}Td(M)\\
&ch(E(M, W, \tau, z))\cdot\varphi(\tau)^{(8)}ch(\mathcal{V}_{\alpha})\Big\}^{(2d)},\nonumber
\end{align}
where
\begin{align}
E(M, W, \tau, z):=c^{2(d-l)}y^{-\frac{l}{2}}\bigotimes^{\infty}_{m=1}\Big(\wedge_{-yq^{m-1}}(W^{*})\otimes\wedge_{-y^{-1}q^{m}}(W)\otimes S_{q^{m}}(T^{*})\otimes S_{q^{m}}(T)\Big)
\end{align}
and $y=e^{2\pi\sqrt{-1}z},$ $c=\prod^{\infty}_{j=1}(1-q^{j}).$
\end{defn}

\begin{lem}
We have
\begin{align}
&Ell(M, W, \mathcal{V}_{\alpha}, \tau, z):=\Big\{\exp{\Big(\frac{1}{24}E_{2}(\tau)\cdot\frac{1}{30}c_{2}(\mathcal{W}_{\alpha})\Big)}\eta(\tau)^{3(d-l)}\prod_{i=1}^{d}\frac{2\pi\sqrt{-1}x_{i}}{\theta(\tau, x_{i})}\\
&\cdot\prod_{j=1}^{l}\theta(\tau, w_{j}-z)\cdot\frac{1}{2}\Big(\prod_{\kappa=1}^{8}\theta_{1}(\tau, y_{\kappa}^{\alpha})+\prod_{\kappa=1}^{8}\theta_{2}(\tau, y_{\kappa}^{\alpha})+\prod_{\kappa=1}^{8}\theta_{3}(\tau, y_{\kappa}^{\alpha})\Big)\Big\}^{(2d)},\nonumber
\end{align}
where
\begin{align}
\eta(\tau):=q^{\frac{1}{24}}\cdot c=q^{\frac{1}{24}}\prod^{\infty}_{j=1}(1-q^{j}).
\end{align}
\end{lem}

\begin{thm}
The generalized elliptic genus $Ell(M, W, \mathcal{V}_{\alpha}, \tau, z)$ is a weak Jacobi form of weight $2d-l+4$ and index $l/2$ provided that $c_{1}(W)=0$ and the first Pontrjagin classes $p_{1}(W)=p_{1}(M).$
\end{thm}
\begin{proof}
According to \cite{C}, we have
\begin{align}
&\eta\Big(-\frac{1}{\tau}\Big)^{3}=\Big(\frac{\tau}{\sqrt{-1}}\Big)^{\frac{3}{2}}\eta(\tau)^{3},~~~~ \eta(\tau+1)^{3}=e^{\frac{\pi\sqrt{-1}}{4}}\eta(\tau)^{3},\\
&\theta(\tau, z+1)=-\theta(\tau, z),~~~~ \theta(\tau, z+\tau)=-q^{-\frac{1}{2}}e^{-2\pi\sqrt{-1}z}\theta(\tau, z).
\end{align}
From Equation (2.28) it is easy to obtain
\begin{align}
&Ell(M, W, \mathcal{V}_{\alpha}, \tau+1, z)=Ell(M, W, \mathcal{V}_{\alpha}, \tau, z);\\
&Ell(M, W, \mathcal{V}_{\alpha}, \tau, z+1)=(-1)^{l}Ell(M, W, \mathcal{V}_{\alpha}, \tau, z);\\
&Ell(M, W, \mathcal{V}_{\alpha}, \tau, z+\tau)=(-1)^{l}\exp{(-\pi\sqrt{-1}l(\tau+2z))}Ell(M, W, \mathcal{V}_{\alpha}, \tau, z);\\
&Ell\Big(M, W, \mathcal{V}_{\alpha}, -\frac{1}{\tau}, \frac{z}{\tau}\Big)=\tau^{2d-l+4}\exp{\Big(\pi\sqrt{-1}l\frac{z^{2}}{\tau}\Big)}Ell(M, W, \mathcal{V}_{\alpha}, \tau, z).
\end{align}
By Lemma 2.1, we can get Theorem 2.4.
\end{proof}

\begin{defn}
The generalized elliptic genus of $(M, J)$ with respect to $W,$ which we denote by $Ell(M, W, \mathcal{V}_{\alpha}, \mathcal{V}_{\beta}, \tau, z)$ is defined by
\begin{align}
Ell(M, W, \mathcal{V}_{\alpha}, \mathcal{V}_{\beta}, \tau, z):=&\Big\{\exp{\Big(\frac{1}{24}E_{2}(\tau)\cdot\frac{1}{30}(c_{2}(\mathcal{W}_{\alpha})+c_{2}(\mathcal{W}_{\beta}))+\frac{c_{1}(W)-c_{1}(M)}{2}\Big)}\\
&Td(M)ch(E(M, W, \tau, z))\cdot\varphi(\tau)^{(16)}ch(\mathcal{V}_{\alpha})ch(\mathcal{V}_{\beta})\Big\}^{(2d)}.\nonumber
\end{align}
\end{defn}

\begin{lem}
We have
\begin{align}
&Ell(M, W, \mathcal{V}_{\alpha}, \mathcal{V}_{\beta}, \tau, z):=\Big\{\exp{\Big(\frac{1}{24}E_{2}(\tau)\cdot\frac{1}{30}(c_{2}(\mathcal{W}_{\alpha})+c_{2}(\mathcal{W}_{\beta}))\Big)}\eta(\tau)^{3(d-l)}\\
&\cdot\prod_{i=1}^{d}\frac{2\pi\sqrt{-1}x_{i}}{\theta(\tau, x_{i})}\prod_{j=1}^{l}\theta(\tau, w_{j}-z)\cdot\frac{1}{4}\Big(\prod_{\kappa=1}^{8}\theta_{1}(\tau, y_{\kappa}^{\alpha})+\prod_{\kappa=1}^{8}\theta_{2}(\tau, y_{\kappa}^{\alpha})+\prod_{\kappa=1}^{8}\theta_{3}(\tau, y_{\kappa}^{\alpha})\Big)\nonumber\\
&\cdot\Big(\prod_{\kappa=1}^{8}\theta_{1}(\tau, y_{\kappa}^{\beta})+\prod_{\kappa=1}^{8}\theta_{2}(\tau, y_{\kappa}^{\beta})+\prod_{\kappa=1}^{8}\theta_{3}(\tau, y_{\kappa}^{\beta})\Big)\Big\}^{(2d)}.\nonumber
\end{align}
\end{lem}

\begin{thm}
The generalized elliptic genus $Ell(M, W, \mathcal{V}_{\alpha}, \mathcal{V}_{\beta}, \tau, z)$ is a weak Jacobi form of weight $2d-l+8$ and index $l/2$ provided that $c_{1}(W)=0$ and the first Pontrjagin classes $p_{1}(W)=p_{1}(M).$
\end{thm}
\begin{proof}
We obtain that
\begin{align}
&Ell(M, W, \mathcal{V}_{\alpha}, \mathcal{V}_{\beta}, \tau+1, z)=Ell(M, W, \mathcal{V}_{\alpha}, \mathcal{V}_{\beta}, \tau, z);\\
&Ell(M, W, \mathcal{V}_{\alpha}, \mathcal{V}_{\beta}, \tau, z+1)=(-1)^{l}Ell(M, W, \mathcal{V}_{\alpha}, \mathcal{V}_{\beta}, \tau, z);\\
&Ell(M, W, \mathcal{V}_{\alpha}, \mathcal{V}_{\beta}, \tau, z+\tau)=(-1)^{l}\exp{(-\pi\sqrt{-1}l(\tau+2z))}Ell(M, W, \mathcal{V}_{\alpha}, \mathcal{V}_{\beta}, \tau, z);\\
&Ell\Big(M, W, \mathcal{V}_{\alpha}, \mathcal{V}_{\beta}, -\frac{1}{\tau}, \frac{z}{\tau}\Big)=\tau^{2d-l+8}\exp{\Big(\pi\sqrt{-1}l\frac{z^{2}}{\tau}\Big)}Ell(M, W, \mathcal{V}_{\alpha}, \mathcal{V}_{\beta}, \tau, z).
\end{align}
Similar to Theorem 2.4, we have Theorem 2.7.
\end{proof}

\section{ Anomaly cancellation formulas for the gauge group $E_{8}$ }

Recall that the Eisenstein series $G_{2k}(\tau)$ are defined to be
\begin{align}
G_{2k}(\tau):=-\frac{B_{2k}}{4k}+\sum_{n=0}^{\infty}\sigma_{2k-1}(n)\cdot q^{n},
\end{align}
where $\sigma_{k}(n):=\sum_{t>0,t\mid n} t^{k}$ and $B_{2k}$ are the Bernoulli numbers.
It is well known that the whole grading ring of modular forms over $SL_{2}(\mathbf{Z})$ are generated by $G_{4}(\tau)$ and $G_{6}(\tau).$

\begin{lem}(\cite{L2,W})
Suppose a function $\phi(\tau, z):\mathcal{H}\times \mathcal{C}\rightarrow\mathcal{C}$ satisfies
\begin{align}
\phi\Big(\frac{a\tau+b}{c\tau+d_{0}}, \frac{z}{c\tau+d_{0}}\Big)=(c\tau+d_{0})^{k}\exp{(\frac{2\pi\sqrt{-1}tcz^{2}}{c\tau+d_{0}})}\phi(\tau, z), ~\Big(\begin{array}{cc}
a & b\\
c & d_{0}
\end{array}\Big)\in SL_{2}(\mathbf{Z}).
\end{align}
We define 
\begin{align}
\Phi(\tau, z):=\exp{(-8\pi^{2}tG_{2}(\tau)z^{2})}\phi(\tau, z):=\sum_{n\geq 0}a_{n}(\tau)\cdot z^{n},
\end{align}
then these $a_{n}(\tau)$ are modular forms of weight $k+n$ over $SL_{2}(\mathbf{Z}).$
\end{lem}

\begin{prop}
Let $c_{1}(W)=c_{1}(M)=0$ and the first Pontrjagin classes $p_{1}(W)=p_{1}(M),$ then the series $\overline{a_{n}}(M, W, \tau)$ determined by
\begin{align}
\exp{(-4\pi^{2}lG_{2}(\tau)z^{2})}Ell(M, W, \mathcal{V}_{\alpha}, \tau, z)=\sum_{n\geq 0}\overline{a_{n}}(M, W, \tau)\cdot z^{n},
\end{align}
are modular forms of weight $2d-l+4+n$ over $SL_{2}(\mathbf{Z}).$ 
Furthermore, the first five series of $\overline{a_{n}}(M, W, \tau)$ are of the following form:
\begin{align}
&\overline{a_{0}}(M, W, \tau)=\exp{\Big(\frac{1}{24}\cdot\frac{1}{30}c_{2}(\mathcal{W}_{\alpha})\Big)}\{Td(M)ch(\wedge_{-1}(W^{*}))\}^{(2d)}+\exp{\Big(\frac{1}{24}\cdot\frac{1}{30}c_{2}(\mathcal{W}_{\alpha})\Big)}\\
&\cdot\Big\{Td(M)\Big(-8-\frac{1}{30}c_{2}(\mathcal{W}_{\alpha})+ch(\mathcal{W}_{\alpha})\Big)ch(\wedge_{-1}(W^{*}))+Td(M)ch(\wedge_{-1}(W^{*}))ch(A_{1})\Big\}^{(2d)}q\nonumber\\
&+\exp{\Big(\frac{1}{24}\cdot\frac{1}{30}c_{2}(\mathcal{W}_{\alpha})\Big)}\Big\{Td(M)\Big(20+\frac{1}{6}c_{2}(\mathcal{W}_{\alpha})+\frac{1}{1800}c_{2}(\mathcal{W}_{\alpha})^{2}+ch(\overline{\mathcal{W}_{\alpha}})-8ch(\mathcal{W}_{\alpha})\nonumber\\
&-\frac{1}{30}c_{2}(\mathcal{W}_{\alpha})ch(\mathcal{W}_{\alpha})\Big)ch(\wedge_{-1}(W^{*}))+Td(M)\Big(-8-\frac{1}{30}c_{2}(\mathcal{W}_{\alpha})+ch(\mathcal{W}_{\alpha})\Big)ch(\wedge_{-1}(W^{*}))\nonumber\\
&\cdot ch(A_{1})+Td(M)ch(\wedge_{-1}(W^{*}))ch(A_{2})\Big\}^{(2d)}q^{2}+\cdot\cdot\cdot,\nonumber
\end{align}
where
\begin{align}
A_{1}&=-W^{*}-W+T^{*}+T-2(d-l),\\
A_{2}&=-W^{*}-W+\wedge^{2}(W^{*})+\wedge^{2}(W)+W^{*}\otimes W-W^{*}\otimes T^{*}\\
&-W^{*}\otimes T-W\otimes T^{*}-W\otimes T+T^{*}+T+S^{2}(T^{*})+S^{2}(T)\nonumber\\
&+T^{*}\otimes T+(d-l)(2d-2l-3);\nonumber
\end{align}
\begin{align}
&\overline{a_{1}}(M, W, \tau)=\exp{\Big(\frac{1}{24}\cdot\frac{1}{30}c_{2}(\mathcal{W}_{\alpha})\Big)}\Big\{2\pi\sqrt{-1}Td(M)ch\Big(\sum_{\rho=0}^{l}(-1)^{\rho}\Big(\rho-\frac{l}{2}\Big)\wedge^{\rho}(W^{*})\Big)\Big\}^{(2d)}\\
&+\exp{\Big(\frac{1}{24}\cdot\frac{1}{30}c_{2}(\mathcal{W}_{\alpha})\Big)}\Big\{2\pi\sqrt{-1}Td(M)\Big(-8-\frac{1}{30}c_{2}(\mathcal{W}_{\alpha})+ch(\mathcal{W}_{\alpha})\Big)ch\Big(\sum_{\rho=0}^{l}(-1)^{\rho}\nonumber\\
&\cdot\Big(\rho-\frac{l}{2}\Big)\wedge^{\rho}(W^{*})\Big)+Td(M)ch(A_{3})\Big\}^{(2d)}q+\cdot\cdot\cdot,\nonumber
\end{align}
where
\begin{align}
A_{3}&=2\pi\sqrt{-1}\sum_{\rho=0}^{l}(-1)^{\rho}\rho\wedge^{\rho}(W^{*})\otimes[-W^{*}-W+T^{*}+T-2(d-l)]\\
&+\sum_{\rho=0}^{l}(-1)^{\rho}\wedge^{\rho}(W^{*})\otimes[-l\pi\sqrt{-1}(-W^{*}-W+T^{*}+T-2(d-l))\nonumber\\
&-2\pi\sqrt{-1}(W^{*}-W)];\nonumber
\end{align}
\begin{align}
&\overline{a_{2}}(M, W, \tau)=\exp{\Big(\frac{1}{24}\cdot\frac{1}{30}c_{2}(\mathcal{W}_{\alpha})\Big)}\Big\{-2\pi^{2}Td(M)ch\Big(\sum_{\rho=0}^{l}(-1)^{\rho}\Big(\rho-\frac{l}{2}\Big)^{2}\wedge^{\rho}(W^{*})\Big)\Big\}^{(2d)}\\
&+\exp{\Big(\frac{1}{24}\cdot\frac{1}{30}c_{2}(\mathcal{W}_{\alpha})\Big)}\Big\{\frac{l}{6}\pi^{2}Td(M)ch\Big(\sum_{\rho=0}^{l}(-1)^{\rho}\wedge^{\rho}(W^{*})\Big)\Big\}^{(2d)}+\cdot\cdot\cdot;\nonumber
\end{align}
\begin{align}
&\overline{a_{3}}(M, W, \tau)=\exp{\Big(\frac{1}{24}\cdot\frac{1}{30}c_{2}(\mathcal{W}_{\alpha})\Big)}\Big\{-\frac{4}{3}\pi^{3}\sqrt{-1}Td(M)ch\Big(\sum_{\rho=0}^{l}(-1)^{\rho}\Big(\rho-\frac{l}{2}\Big)^{3}\\
&\cdot\wedge^{\rho}(W^{*})\Big)\Big\}^{(2d)}+\exp{\Big(\frac{1}{24}\cdot\frac{1}{30}c_{2}(\mathcal{W}_{\alpha})\Big)}\Big\{\frac{l}{3}\pi^{3}\sqrt{-1}Td(M)ch\Big(\sum_{\rho=0}^{l}(-1)^{\rho}\Big(\rho-\frac{l}{2}\Big)\nonumber\\
&\cdot\wedge^{\rho}(W^{*})\Big)\Big\}^{(2d)}+\cdot\cdot\cdot;\nonumber
\end{align}
\begin{align}
&\overline{a_{4}}(M, W, \tau)=\exp{\Big(\frac{1}{24}\cdot\frac{1}{30}c_{2}(\mathcal{W}_{\alpha})\Big)}\Big\{\frac{2}{3}\pi^{4}Td(M)ch\Big(\sum_{\rho=0}^{l}(-1)^{\rho}\Big(\rho-\frac{l}{2}\Big)^{4}\wedge^{\rho}(W^{*})\Big)\Big\}^{(2d)}\\
&+\exp{\Big(\frac{1}{24}\cdot\frac{1}{30}c_{2}(\mathcal{W}_{\alpha})\Big)}\Big\{-\frac{l}{3}\pi^{4}Td(M)ch\Big(\sum_{\rho=0}^{l}(-1)^{\rho}\Big(\rho-\frac{l}{2}\Big)^{2}\wedge^{\rho}(W^{*})\Big)\Big\}^{(2d)}\nonumber\\
&+\exp{\Big(\frac{1}{24}\cdot\frac{1}{30}c_{2}(\mathcal{W}_{\alpha})\Big)}\Big\{\frac{l^{2}}{72}\pi^{4}Td(M)ch\Big(\sum_{\rho=0}^{l}(-1)^{\rho}\wedge^{\rho}(W^{*})\Big)\Big\}^{(2d)}+\cdot\cdot\cdot.\nonumber
\end{align}
\end{prop}
\begin{proof}
We check at once that
\begin{align}
&\overline{a_{0}}(M, W, \tau)=Ell(M, W, \mathcal{V}_{\alpha}, \tau, 0).
\end{align}
For $c_{1}(W)=c_{1}(M)=0$ and $p_{1}(W)=p_{1}(M),$ a straightforward calculation shows that
\begin{align}
Ell(M, W, \mathcal{V}_{\alpha}, \tau, 0)=&\Big\{\exp{\Big(\frac{1}{24}E_{2}(\tau)\cdot\frac{1}{30}c_{2}(\mathcal{W}_{\alpha})\Big)}Td(M)ch(E(M, W, \tau, 0))\\
&\cdot\varphi(\tau)^{(8)}ch(\mathcal{V}_{\alpha})\Big\}^{(2d)},\nonumber
\end{align}
where
\begin{align}
E(M, W, \tau, 0)&=\wedge_{-1}(W^{*})\{1+[-W^{*}-W+T^{*}+T-2(d-l)]q+[-W^{*}-W+\wedge^{2}(W^{*})\\
&+\wedge^{2}(W)+W^{*}\otimes W-W^{*}\otimes T^{*}-W^{*}\otimes T-W\otimes T^{*}-W\otimes T+T^{*}+T\nonumber\\
&+S^{2}(T^{*})+S^{2}(T)+T^{*}\otimes T+(d-l)(2d-2l-3)]q^{2}+\cdot\cdot\cdot\}.\nonumber
\end{align}
Fix
\begin{align}
Ell(M, W, \mathcal{V}_{\alpha}, \tau, z):=\overline{B_{0}}(z)+\overline{B_{1}}(z)q+\overline{B_{2}}(z)q^{2}+\cdot\cdot\cdot,
\end{align}
and
\begin{align}
\exp{(-4\pi^{2}lG_{2}(\tau)z^{2})}:=C_{0}(z)+C_{1}(z)q+C_{2}(z)q^{2}+\cdot\cdot\cdot.
\end{align}
A long but straightforward calculation shows that
\begin{align}
&\overline{B_{0}}(z)=\exp{\Big(\frac{1}{24}\cdot\frac{1}{30}c_{2}(\mathcal{W}_{\alpha})\Big)}\Big\{Td(M)ch\Big(\sum_{\rho=0}^{l}(-1)^{\rho}\wedge^{\rho}(W^{*})\Big)\Big\}^{(2d)}\\
&+\exp{\Big(\frac{1}{24}\cdot\frac{1}{30}c_{2}(\mathcal{W}_{\alpha})\Big)}\Big\{2\pi\sqrt{-1}Td(M)ch\Big(\sum_{\rho=0}^{l}(-1)^{\rho}\Big(\rho-\frac{l}{2}\Big)\wedge^{\rho}(W^{*})\Big)\Big\}^{(2d)}z\nonumber\\
&+\exp{\Big(\frac{1}{24}\cdot\frac{1}{30}c_{2}(\mathcal{W}_{\alpha})\Big)}\Big\{-2\pi^{2}Td(M)ch\Big(\sum_{\rho=0}^{l}(-1)^{\rho}\Big(\rho-\frac{l}{2}\Big)^{2}\wedge^{\rho}(W^{*})\Big)\Big\}^{(2d)}z^{2}\nonumber\\
&+\exp{\Big(\frac{1}{24}\cdot\frac{1}{30}c_{2}(\mathcal{W}_{\alpha})\Big)}\Big\{-\frac{4}{3}\pi^{3}\sqrt{-1}Td(M)ch\Big(\sum_{\rho=0}^{l}(-1)^{\rho}\Big(\rho-\frac{l}{2}\Big)^{3}\wedge^{\rho}(W^{*})\Big)\Big\}^{(2d)}z^{3}\nonumber\\
&+\exp{\Big(\frac{1}{24}\cdot\frac{1}{30}c_{2}(\mathcal{W}_{\alpha})\Big)}\Big\{\frac{2}{3}\pi^{4}Td(M)ch\Big(\sum_{\rho=0}^{l}(-1)^{\rho}\Big(\rho-\frac{l}{2}\Big)^{4}\wedge^{\rho}(W^{*})\Big)\Big\}^{(2d)}z^{4}+\cdot\cdot\cdot;\nonumber
\end{align}
\begin{align}
&\overline{B_{1}}(z)=\exp{\Big(\frac{1}{24}\cdot\frac{1}{30}c_{2}(\mathcal{W}_{\alpha})\Big)}\Big\{Td(M)\Big(-8-\frac{1}{30}c_{2}(\mathcal{W}_{\alpha})+ch(\mathcal{W}_{\alpha})\Big)ch\Big(\sum_{\rho=0}^{l}(-1)^{\rho}
\end{align}
\begin{align}
&\cdot\wedge^{\rho}(W^{*})\Big)+Td(M)ch\Big(\sum_{\rho=0}^{l}(-1)^{\rho}\wedge^{\rho}(W^{*})\otimes[-W^{*}-W+T^{*}+T-2(d-l)]\Big)\Big\}^{(2d)}\nonumber\\
&+\exp{\Big(\frac{1}{24}\cdot\frac{1}{30}c_{2}(\mathcal{W}_{\alpha})\Big)}\Big\{2\pi\sqrt{-1}Td(M)\Big(-8-\frac{1}{30}c_{2}(\mathcal{W}_{\alpha})+ch(\mathcal{W}_{\alpha})\Big)ch\Big(\sum_{\rho=0}^{l}(-1)^{\rho}\nonumber\\
&\cdot\Big(\rho-\frac{l}{2}\Big)\wedge^{\rho}(W^{*})\Big)+Td(M)ch\Big(2\pi\sqrt{-1}\sum_{\rho=0}^{l}(-1)^{\rho}\rho\wedge^{\rho}(W^{*})\otimes[-W^{*}-W+T^{*}\nonumber\\
&+T-2(d-l)]+\sum_{\rho=0}^{l}(-1)^{\rho}\wedge^{\rho}(W^{*})\otimes[-l\pi\sqrt{-1}(-W^{*}-W+T^{*}+T-2(d-l))\nonumber\\
&-2\pi\sqrt{-1}(W^{*}-W)]\Big)\Big\}^{(2d)}z+\cdot\cdot\cdot;\nonumber
\end{align}
\begin{align}
C_{0}(z)=1+\frac{l}{6}\pi^{2}z^{2}+\frac{l^{2}}{72}\pi^{4}z^{4}+\cdot\cdot\cdot;
\end{align}
\begin{align}
C_{1}(z)=-4l\pi^{2}z^{2}-\frac{2l^{2}}{3}\pi^{4}z^{4}+\cdot\cdot\cdot;
\end{align}
\begin{align}
C_{2}(z)=-12l\pi^{2}z^{2}+14l^{2}\pi^{4}z^{4}+\cdot\cdot\cdot.
\end{align}
We can rewrite (3.4) as
\begin{align}
\sum_{n\geq 0}\overline{a_{n}}(M, W, \tau)\cdot z^{n}=\overline{B_{0}}(z)C_{0}(z)+[\overline{B_{0}}(z)C_{1}(z)+\overline{B_{1}}(z)C_{0}(z)]q+\cdot\cdot\cdot,
\end{align}
by the analysis above, we can have Proposition 3.2.
\end{proof}

\begin{thm}
Let $c_{1}(W)=c_{1}(M)=0$ and the first Pontrjagin classes $p_{1}(W)=p_{1}(M),$ then\\
1) if either $2d<16,$ $2d-l$ is odd or $2d<16, 2d-l\leq-2$ but $2d-l\neq-4,$ then
\begin{align}
&\{Td(M)ch(\wedge_{-1}(W^{*}))\}^{(2d)}=0;\\
&\Big\{Td(M)\Big(-8-\frac{1}{30}c_{2}(\mathcal{W}_{\alpha})+ch(\mathcal{W}_{\alpha})\Big)ch(\wedge_{-1}(W^{*}))+Td(M)ch(\wedge_{-1}(W^{*}))ch(A_{1})\Big\}^{(2d)}\\
&=0;\nonumber\\
&\Big\{Td(M)\Big(20+\frac{1}{6}c_{2}(\mathcal{W}_{\alpha})+\frac{1}{1800}c_{2}(\mathcal{W}_{\alpha})^{2}+ch(\overline{\mathcal{W}_{\alpha}})-8ch(\mathcal{W}_{\alpha})-\frac{1}{30}c_{2}(\mathcal{W}_{\alpha})ch(\mathcal{W}_{\alpha})\Big)
\end{align}
\begin{align}
&\cdot ch(\wedge_{-1}(W^{*}))+Td(M)\Big(-8-\frac{1}{30}c_{2}(\mathcal{W}_{\alpha})+ch(\mathcal{W}_{\alpha})\Big)ch(\wedge_{-1}(W^{*}))ch(A_{1})+Td(M)\nonumber\\
&\cdot ch(\wedge_{-1}(W^{*}))ch(A_{2})\Big\}^{(2d)}=0,\nonumber
\end{align}
2) if either $2d<16,$ $2d-l$ is even or $2d<16, 2d-l\leq-3$ but $2d-l\neq-5,$ then
\begin{align}
&\Big\{2\pi\sqrt{-1}Td(M)ch\Big(\sum_{\rho=0}^{l}(-1)^{\rho}\Big(\rho-\frac{l}{2}\Big)\wedge^{\rho}(W^{*})\Big)\Big\}^{(2d)}=0;\\
&\Big\{2\pi\sqrt{-1}Td(M)\Big(-8-\frac{1}{30}c_{2}(\mathcal{W}_{\alpha})+ch(\mathcal{W}_{\alpha})\Big)ch\Big(\sum_{\rho=0}^{l}(-1)^{\rho}\Big(\rho-\frac{l}{2}\Big)\wedge^{\rho}(W^{*})\Big)\\
&+Td(M)ch(A_{3})\Big\}^{(2d)}=0,\nonumber
\end{align}
3) if either $2d<16,$ $2d-l$ is odd or $2d<16, 2d-l\leq-4$ but $2d-l\neq-6,$ then
\begin{align}
&\Big\{-2\pi^{2}Td(M)ch\Big(\sum_{\rho=0}^{l}(-1)^{\rho}\Big(\rho-\frac{l}{2}\Big)^{2}\wedge^{\rho}(W^{*})\Big)\Big\}^{(2d)}+\Big\{\frac{l}{6}\pi^{2}Td(M)\\
&\cdot ch\Big(\sum_{\rho=0}^{l}(-1)^{\rho}\wedge^{\rho}(W^{*})\Big)\Big\}^{(2d)}=0,\nonumber
\end{align}
4) if either $2d<16,$ $2d-l$ is even or $2d<16, 2d-l\leq-5$ but $2d-l\neq-7,$ then
\begin{align}
&\Big\{-\frac{4}{3}\pi^{3}\sqrt{-1}Td(M)ch\Big(\sum_{\rho=0}^{l}(-1)^{\rho}\Big(\rho-\frac{l}{2}\Big)^{3}\wedge^{\rho}(W^{*})\Big)\Big\}^{(2d)}+\Big\{\frac{l}{3}\pi^{3}\sqrt{-1}Td(M)\\
&\cdot ch\Big(\sum_{\rho=0}^{l}(-1)^{\rho}\Big(\rho-\frac{l}{2}\Big)\wedge^{\rho}(W^{*})\Big)\Big\}^{(2d)}=0,\nonumber
\end{align}
5) if either $2d<16,$ $2d-l$ is odd or $2d<16, 2d-l\leq-6$ but $2d-l\neq-8,$ then
\begin{align}
&\Big\{\frac{2}{3}\pi^{4}Td(M)ch\Big(\sum_{\rho=0}^{l}(-1)^{\rho}\Big(\rho-\frac{l}{2}\Big)^{4}\wedge^{\rho}(W^{*})\Big)\Big\}^{(2d)}+\Big\{-\frac{l}{3}\pi^{4}Td(M)ch\Big(\sum_{\rho=0}^{l}(-1)^{\rho}\\
&\cdot\Big(\rho-\frac{l}{2}\Big)^{2}\wedge^{\rho}(W^{*})\Big)\Big\}^{(2d)}+\Big\{\frac{l^{2}}{72}\pi^{4}Td(M)ch\Big(\sum_{\rho=0}^{l}(-1)^{\rho}\wedge^{\rho}(W^{*})\Big)\Big\}^{(2d)}=0.\nonumber
\end{align}
\end{thm}

\begin{thm}
Let $c_{1}(W)=c_{1}(M)=0$ and the first Pontrjagin classes $p_{1}(W)=p_{1}(M),$ then\\
1) if $2d-l=0,$ then
\begin{align}
&\Big\{Td(M)\Big(-8-\frac{1}{30}c_{2}(\mathcal{W}_{\alpha})+ch(\mathcal{W}_{\alpha})\Big)ch(\wedge_{-1}(W^{*}))+Td(M)ch(\wedge_{-1}(W^{*}))ch(A_{1})\Big\}^{(2d)}\\
&=240\{Td(M)ch(\wedge_{-1}(W^{*}))\}^{(2d)};\nonumber\\
&\Big\{Td(M)\Big(20+\frac{1}{6}c_{2}(\mathcal{W}_{\alpha})+\frac{1}{1800}c_{2}(\mathcal{W}_{\alpha})^{2}+ch(\overline{\mathcal{W}_{\alpha}})-8ch(\mathcal{W}_{\alpha})-\frac{1}{30}c_{2}(\mathcal{W}_{\alpha})ch(\mathcal{W}_{\alpha})\Big)\\
&\cdot ch(\wedge_{-1}(W^{*}))+Td(M)\Big(-8-\frac{1}{30}c_{2}(\mathcal{W}_{\alpha})+ch(\mathcal{W}_{\alpha})\Big)ch(\wedge_{-1}(W^{*}))ch(A_{1})+Td(M)\nonumber\\
&\cdot ch(\wedge_{-1}(W^{*}))ch(A_{2})\Big\}^{(2d)}=2160\{Td(M)ch(\wedge_{-1}(W^{*}))\}^{(2d)},\nonumber
\end{align}
where
\begin{align}
A_{1}&=-W^{*}-W+T^{*}+T+l;\\
A_{2}&=-W^{*}-W+\wedge^{2}(W^{*})+\wedge^{2}(W)+W^{*}\otimes W-W^{*}\otimes T^{*}\\
&-W^{*}\otimes T-W\otimes T^{*}-W\otimes T+T^{*}+T+S^{2}(T^{*})+S^{2}(T)\nonumber\\
&+T^{*}\otimes T-(d-l)(3+l),\nonumber
\end{align}
2) if $2d-l=2,$ then
\begin{align}
&\Big\{Td(M)\Big(-8-\frac{1}{30}c_{2}(\mathcal{W}_{\alpha})+ch(\mathcal{W}_{\alpha})\Big)ch(\wedge_{-1}(W^{*}))+Td(M)ch(\wedge_{-1}(W^{*}))ch(A_{1})\Big\}^{(2d)}\\
&=-504\{Td(M)ch(\wedge_{-1}(W^{*}))\}^{(2d)};\nonumber\\
&\Big\{Td(M)\Big(20+\frac{1}{6}c_{2}(\mathcal{W}_{\alpha})+\frac{1}{1800}c_{2}(\mathcal{W}_{\alpha})^{2}+ch(\overline{\mathcal{W}_{\alpha}})-8ch(\mathcal{W}_{\alpha})-\frac{1}{30}c_{2}(\mathcal{W}_{\alpha})ch(\mathcal{W}_{\alpha})\Big)\\
&\cdot ch(\wedge_{-1}(W^{*}))+Td(M)\Big(-8-\frac{1}{30}c_{2}(\mathcal{W}_{\alpha})+ch(\mathcal{W}_{\alpha})\Big)ch(\wedge_{-1}(W^{*}))ch(A_{1})+Td(M)\nonumber\\
&\cdot ch(\wedge_{-1}(W^{*}))ch(A_{2})\Big\}^{(2d)}=-16632\{Td(M)ch(\wedge_{-1}(W^{*}))\}^{(2d)},\nonumber
\end{align}
where
\begin{align}
A_{1}&=-W^{*}-W+T^{*}+T+l-2;\\
A_{2}&=-W^{*}-W+\wedge^{2}(W^{*})+\wedge^{2}(W)+W^{*}\otimes W-W^{*}\otimes T^{*}\\
&-W^{*}\otimes T-W\otimes T^{*}-W\otimes T+T^{*}+T+S^{2}(T^{*})+S^{2}(T)\nonumber\\
&+T^{*}\otimes T-(d-l)(1+l),\nonumber
\end{align}
3) if $2d-l=4,$ then
\begin{align}
&\Big\{Td(M)\Big(-8-\frac{1}{30}c_{2}(\mathcal{W}_{\alpha})+ch(\mathcal{W}_{\alpha})\Big)ch(\wedge_{-1}(W^{*}))+Td(M)ch(\wedge_{-1}(W^{*}))ch(A_{1})\Big\}^{(2d)}\\
&=480\{Td(M)ch(\wedge_{-1}(W^{*}))\}^{(2d)};\nonumber\\
&\Big\{Td(M)\Big(20+\frac{1}{6}c_{2}(\mathcal{W}_{\alpha})+\frac{1}{1800}c_{2}(\mathcal{W}_{\alpha})^{2}+ch(\overline{\mathcal{W}_{\alpha}})-8ch(\mathcal{W}_{\alpha})-\frac{1}{30}c_{2}(\mathcal{W}_{\alpha})ch(\mathcal{W}_{\alpha})\Big)\\
&\cdot ch(\wedge_{-1}(W^{*}))+Td(M)\Big(-8-\frac{1}{30}c_{2}(\mathcal{W}_{\alpha})+ch(\mathcal{W}_{\alpha})\Big)ch(\wedge_{-1}(W^{*}))ch(A_{1})+Td(M)\nonumber\\
&\cdot ch(\wedge_{-1}(W^{*}))ch(A_{2})\Big\}^{(2d)}=61920\{Td(M)ch(\wedge_{-1}(W^{*}))\}^{(2d)},\nonumber
\end{align}
where
\begin{align}
A_{1}&=-W^{*}-W+T^{*}+T+l-4;\\
A_{2}&=-W^{*}-W+\wedge^{2}(W^{*})+\wedge^{2}(W)+W^{*}\otimes W-W^{*}\otimes T^{*}\\
&-W^{*}\otimes T-W\otimes T^{*}-W\otimes T+T^{*}+T+S^{2}(T^{*})+S^{2}(T)\nonumber\\
&+T^{*}\otimes T+(d-l)(1-l),\nonumber
\end{align}
4) if $2d-l=6,$ then
\begin{align}
&\Big\{Td(M)\Big(-8-\frac{1}{30}c_{2}(\mathcal{W}_{\alpha})+ch(\mathcal{W}_{\alpha})\Big)ch(\wedge_{-1}(W^{*}))+Td(M)ch(\wedge_{-1}(W^{*}))ch(A_{1})\Big\}^{(2d)}\\
&=-264\{Td(M)ch(\wedge_{-1}(W^{*}))\}^{(2d)};\nonumber\\
&\Big\{Td(M)\Big(20+\frac{1}{6}c_{2}(\mathcal{W}_{\alpha})+\frac{1}{1800}c_{2}(\mathcal{W}_{\alpha})^{2}+ch(\overline{\mathcal{W}_{\alpha}})-8ch(\mathcal{W}_{\alpha})-\frac{1}{30}c_{2}(\mathcal{W}_{\alpha})ch(\mathcal{W}_{\alpha})\Big)\\
&\cdot ch(\wedge_{-1}(W^{*}))+Td(M)\Big(-8-\frac{1}{30}c_{2}(\mathcal{W}_{\alpha})+ch(\mathcal{W}_{\alpha})\Big)ch(\wedge_{-1}(W^{*}))ch(A_{1})+Td(M)\nonumber\\
&\cdot ch(\wedge_{-1}(W^{*}))ch(A_{2})\Big\}^{(2d)}=-135432\{Td(M)ch(\wedge_{-1}(W^{*}))\}^{(2d)},\nonumber
\end{align}
where
\begin{align}
A_{1}&=-W^{*}-W+T^{*}+T+l-6;\\
A_{2}&=-W^{*}-W+\wedge^{2}(W^{*})+\wedge^{2}(W)+W^{*}\otimes W-W^{*}\otimes T^{*}\\
&-W^{*}\otimes T-W\otimes T^{*}-W\otimes T+T^{*}+T+S^{2}(T^{*})+S^{2}(T)\nonumber\\
&+T^{*}\otimes T+(d-l)(3-l),\nonumber
\end{align}
5) if $2d-l=8,$ then
\begin{align}
&\Big\{Td(M)\Big(20+\frac{1}{6}c_{2}(\mathcal{W}_{\alpha})+\frac{1}{1800}c_{2}(\mathcal{W}_{\alpha})^{2}+ch(\overline{\mathcal{W}_{\alpha}})-8ch(\mathcal{W}_{\alpha})-\frac{1}{30}c_{2}(\mathcal{W}_{\alpha})ch(\mathcal{W}_{\alpha})\Big)\\
&\cdot ch(\wedge_{-1}(W^{*}))+Td(M)\Big(-8-\frac{1}{30}c_{2}(\mathcal{W}_{\alpha})+ch(\mathcal{W}_{\alpha})\Big)ch(\wedge_{-1}(W^{*}))ch(A_{1})+Td(M)\nonumber\\
&\cdot ch(\wedge_{-1}(W^{*}))ch(A_{2})\Big\}^{(2d)}=196560\{Td(M)ch(\wedge_{-1}(W^{*}))\}^{(2d)}-24\Big\{Td(M)\Big(-8\nonumber\\
&-\frac{1}{30}c_{2}(\mathcal{W}_{\alpha})+ch(\mathcal{W}_{\alpha})\Big)ch(\wedge_{-1}(W^{*}))+Td(M)ch(\wedge_{-1}(W^{*}))ch(A_{1})\Big\}^{(2d)},\nonumber
\end{align}
where
\begin{align}
A_{1}&=-W^{*}-W+T^{*}+T+l-8;\\
A_{2}&=-W^{*}-W+\wedge^{2}(W^{*})+\wedge^{2}(W)+W^{*}\otimes W-W^{*}\otimes T^{*}\\
&-W^{*}\otimes T-W\otimes T^{*}-W\otimes T+T^{*}+T+S^{2}(T^{*})+S^{2}(T)\nonumber\\
&+T^{*}\otimes T+(d-l)(5-l),\nonumber
\end{align}
6) if $2d-l=10,$ then
\begin{align}
&\Big\{Td(M)\Big(-8-\frac{1}{30}c_{2}(\mathcal{W}_{\alpha})+ch(\mathcal{W}_{\alpha})\Big)ch(\wedge_{-1}(W^{*}))+Td(M)ch(\wedge_{-1}(W^{*}))ch(A_{1})\Big\}^{(2d)}\\
&=-24\{Td(M)ch(\wedge_{-1}(W^{*}))\}^{(2d)};\nonumber\\
&\Big\{Td(M)\Big(20+\frac{1}{6}c_{2}(\mathcal{W}_{\alpha})+\frac{1}{1800}c_{2}(\mathcal{W}_{\alpha})^{2}+ch(\overline{\mathcal{W}_{\alpha}})-8ch(\mathcal{W}_{\alpha})-\frac{1}{30}c_{2}(\mathcal{W}_{\alpha})ch(\mathcal{W}_{\alpha})\Big)\\
&\cdot ch(\wedge_{-1}(W^{*}))+Td(M)\Big(-8-\frac{1}{30}c_{2}(\mathcal{W}_{\alpha})+ch(\mathcal{W}_{\alpha})\Big)ch(\wedge_{-1}(W^{*}))ch(A_{1})+Td(M)\nonumber\\
&\cdot ch(\wedge_{-1}(W^{*}))ch(A_{2})\Big\}^{(2d)}=-196632\{Td(M)ch(\wedge_{-1}(W^{*}))\}^{(2d)},\nonumber
\end{align}
where
\begin{align}
A_{1}&=-W^{*}-W+T^{*}+T+l-10;\\
A_{2}&=-W^{*}-W+\wedge^{2}(W^{*})+\wedge^{2}(W)+W^{*}\otimes W-W^{*}\otimes T^{*}\\
&-W^{*}\otimes T-W\otimes T^{*}-W\otimes T+T^{*}+T+S^{2}(T^{*})+S^{2}(T)\nonumber\\
&+T^{*}\otimes T+(d-l)(7-l),\nonumber
\end{align}
7) if $2d-l=12,$ then
\begin{align}
&\Big\{Td(M)\Big(20+\frac{1}{6}c_{2}(\mathcal{W}_{\alpha})+\frac{1}{1800}c_{2}(\mathcal{W}_{\alpha})^{2}+ch(\overline{\mathcal{W}_{\alpha}})-8ch(\mathcal{W}_{\alpha})-\frac{1}{30}c_{2}(\mathcal{W}_{\alpha})ch(\mathcal{W}_{\alpha})\Big)\\
&\cdot ch(\wedge_{-1}(W^{*}))+Td(M)\Big(-8-\frac{1}{30}c_{2}(\mathcal{W}_{\alpha})+ch(\mathcal{W}_{\alpha})\Big)ch(\wedge_{-1}(W^{*}))ch(A_{1})+Td(M)\nonumber\\
&\cdot ch(\wedge_{-1}(W^{*}))ch(A_{2})\Big\}^{(2d)}=146880\{Td(M)ch(\wedge_{-1}(W^{*}))\}^{(2d)}+216\Big\{Td(M)\Big(-8\nonumber\\
&-\frac{1}{30}c_{2}(\mathcal{W}_{\alpha})+ch(\mathcal{W}_{\alpha})\Big)ch(\wedge_{-1}(W^{*}))+Td(M)ch(\wedge_{-1}(W^{*}))ch(A_{1})\Big\}^{(2d)},\nonumber
\end{align}
where
\begin{align}
A_{1}&=-W^{*}-W+T^{*}+T+l-12;\\
A_{2}&=-W^{*}-W+\wedge^{2}(W^{*})+\wedge^{2}(W)+W^{*}\otimes W-W^{*}\otimes T^{*}\\
&-W^{*}\otimes T-W\otimes T^{*}-W\otimes T+T^{*}+T+S^{2}(T^{*})+S^{2}(T)\nonumber\\
&+T^{*}\otimes T+(d-l)(9-l).\nonumber
\end{align}
\end{thm}
\begin{proof}
Since $\overline{a_{n}}(M, W, \tau)$ are modular forms of weight $2d-l+4+n$ over $SL_{2}(\mathbf{Z}).$ 
Consequently,\\
1) if $2d-l=0,$ then $\overline{a_{0}}(M, W, \tau)$ is proportional to
\begin{align}
G_{4}(\tau)=1+240q+2160q^{2}+6720q^{3}+\cdot\cdot\cdot,
\end{align}
by comparing the constant term and the coefficients of $q,$ we can get (3.32) and (3.33).\\
2) if $2d-l=2,$ then $\overline{a_{0}}(M, W, \tau)$ is proportional to
\begin{align}
G_{6}(\tau)=1-504q-16632q^{2}-122976q^{3}+\cdot\cdot\cdot,
\end{align}
we have (3.36) and (3.37).\\
3) if $2d-l=4,$ then $\overline{a_{0}}(M, W, \tau)$ is proportional to
\begin{align}
G_{4}(\tau)^{2}=1+480q+61920q^{2}+\cdot\cdot\cdot,
\end{align}
Via simple calculations, we can obtain (3.40) and (3.41).\\
4) if $2d-l=6,$ then $\overline{a_{0}}(M, W, \tau)$ is proportional to
\begin{align}
G_{4}(\tau)G_{6}(\tau)=1-264q-135432q^{2}+\cdot\cdot\cdot,
\end{align}
It is easy to obtain (3.44) and (3.45).\\
5) if $2d-l=8,$ then
\begin{align}
\overline{a_{0}}(M, W, \tau)=\lambda_{1}G_{4}(\tau)^{3}+\lambda_{2}G_{6}(\tau)^{2},
\end{align}
where $\lambda_{1}, \lambda_{2}$ are degree $2d$ forms. It is widely known that
\begin{align}
&G_{4}(\tau)^{3}=1+720q+179280q^{2}+\cdot\cdot\cdot,\\
&G_{6}(\tau)^{2}=1-1008q+220752q^{2}+\cdot\cdot\cdot.
\end{align}
In (3.62), we compare the coefficients of $1, q, q^{2},$ we get three equations about $\lambda_{1}, \lambda_{2}.$
Solve the three equations, we get (3.48).\\
6) if $2d-l=10,$ then $\overline{a_{0}}(M, W, \tau)$ is proportional to
\begin{align}
G_{4}(\tau)^{2}G_{6}(\tau)=1-24q-196632q^{2}+\cdot\cdot\cdot,
\end{align}
so (3.51) and (3.52) holds.\\
7) if $2d-l=12,$ then
\begin{align}
\overline{a_{0}}(M, W, \tau)=\overline{\lambda_{1}}G_{4}(\tau)^{4}+\overline{\lambda_{2}}G_{4}(\tau)G_{6}(\tau)^{2},
\end{align}
where $\overline{\lambda_{1}}, \overline{\lambda_{2}}$ are degree $2d$ forms. Note that
\begin{align}
&G_{4}(\tau)^{4}=1+960q+354240q^{2}+\cdot\cdot\cdot,\\
&G_{4}(\tau)G_{6}(\tau)^{2}=1-768q-19008q^{2}+\cdot\cdot\cdot.
\end{align}
By applying the formula shown in (3.66), we can get (3.55).
\end{proof}

\section{ Anomaly cancellation formulas for the gauge group $E_{8}\times E_{8}$ }

\begin{prop}
Let $c_{1}(W)=c_{1}(M)=0$ and the first Pontrjagin classes $p_{1}(W)=p_{1}(M),$ then the series $\widetilde{a_{n}}(M, W, \tau)$ determined by
\begin{align}
\exp{(-4\pi^{2}lG_{2}(\tau)z^{2})}Ell(M, W, \mathcal{V}_{\alpha}, \mathcal{V}_{\beta}, \tau, z)=\sum_{n\geq 0}\widetilde{a_{n}}(M, W, \tau)\cdot z^{n},
\end{align}
are modular forms of weight $2d-l+8+n$ over $SL_{2}(\mathbf{Z}).$ 
Furthermore, the first five series of $\widetilde{a_{n}}(M, W, \tau)$ are of the following form:
\begin{align}
&\widetilde{a_{0}}(M, W, \tau)=\exp{\Big(\frac{1}{24}\cdot\frac{1}{30}(c_{2}(\mathcal{W}_{\alpha})+c_{2}(\mathcal{W}_{\beta}))\Big)}\{Td(M)ch(\wedge_{-1}(W^{*}))\}^{(2d)}\\
&+\exp{\Big(\frac{1}{24}\cdot\frac{1}{30}(c_{2}(\mathcal{W}_{\alpha})+c_{2}(\mathcal{W}_{\beta}))\Big)}\Big\{Td(M)\Big(-16-\frac{1}{30}(c_{2}(\mathcal{W}_{\alpha})+c_{2}(\mathcal{W}_{\beta}))\nonumber\\
&+ch(\mathcal{W}_{\alpha})+ch(\mathcal{W}_{\beta})\Big)ch(\wedge_{-1}(W^{*}))+Td(M)ch(\wedge_{-1}(W^{*}))ch(A_{1})\Big\}^{(2d)}q\nonumber\\
&+\exp{\Big(\frac{1}{24}\cdot\frac{1}{30}(c_{2}(\mathcal{W}_{\alpha})+c_{2}(\mathcal{W}_{\beta}))\Big)}\Big\{Td(M)\Big(104+\frac{13}{30}(c_{2}(\mathcal{W}_{\alpha})+c_{2}(\mathcal{W}_{\beta}))\nonumber\\
&+\frac{1}{1800}(c_{2}(\mathcal{W}_{\alpha})+c_{2}(\mathcal{W}_{\beta}))^{2}+ch(\overline{\mathcal{W}_{\alpha}})+ch(\overline{\mathcal{W}_{\beta}})-16ch(\mathcal{W}_{\alpha})-\frac{1}{30}(c_{2}(\mathcal{W}_{\alpha})\nonumber\\
&+c_{2}(\mathcal{W}_{\beta}))ch(\mathcal{W}_{\alpha})-16ch(\mathcal{W}_{\beta})-\frac{1}{30}(c_{2}(\mathcal{W}_{\alpha})+c_{2}(\mathcal{W}_{\beta}))ch(\mathcal{W}_{\beta})+ch(\mathcal{W}_{\alpha})\nonumber\\
&\cdot ch(\mathcal{W}_{\beta})\Big)ch(\wedge_{-1}(W^{*}))+Td(M)\Big(-16-\frac{1}{30}(c_{2}(\mathcal{W}_{\alpha})+c_{2}(\mathcal{W}_{\beta}))+ch(\mathcal{W}_{\alpha})\nonumber\\
&+ch(\mathcal{W}_{\beta})\Big)ch(\wedge_{-1}(W^{*}))ch(A_{1})+Td(M)ch(\wedge_{-1}(W^{*}))ch(A_{2})\Big\}^{(2d)}q^{2}+\cdot\cdot\cdot,\nonumber
\end{align}
where
\begin{align}
A_{1}&=-W^{*}-W+T^{*}+T-2(d-l),\\
A_{2}&=-W^{*}-W+\wedge^{2}(W^{*})+\wedge^{2}(W)+W^{*}\otimes W-W^{*}\otimes T^{*}\\
&-W^{*}\otimes T-W\otimes T^{*}-W\otimes T+T^{*}+T+S^{2}(T^{*})+S^{2}(T)\nonumber\\
&+T^{*}\otimes T+(d-l)(2d-2l-3);\nonumber
\end{align}
\begin{align}
&\widetilde{a_{1}}(M, W, \tau)=\exp{\Big(\frac{1}{24}\cdot\frac{1}{30}(c_{2}(\mathcal{W}_{\alpha})+c_{2}(\mathcal{W}_{\beta}))\Big)}\Big\{2\pi\sqrt{-1}Td(M)ch\Big(\sum_{\rho=0}^{l}(-1)^{\rho}\\
&\cdot\Big(\rho-\frac{l}{2}\Big)\wedge^{\rho}(W^{*})\Big)\Big\}^{(2d)}+\exp{\Big(\frac{1}{24}\cdot\frac{1}{30}(c_{2}(\mathcal{W}_{\alpha})+c_{2}(\mathcal{W}_{\beta}))\Big)}\Big\{2\pi\sqrt{-1}Td(M)\nonumber\\
&\cdot\Big(-16-\frac{1}{30}(c_{2}(\mathcal{W}_{\alpha})+c_{2}(\mathcal{W}_{\beta}))+ch(\mathcal{W}_{\alpha})+ch(\mathcal{W}_{\beta})\Big)ch\Big(\sum_{\rho=0}^{l}(-1)^{\rho}\Big(\rho-\frac{l}{2}\Big)\nonumber\\
&\cdot\wedge^{\rho}(W^{*})\Big)+Td(M)ch(A_{3})\Big\}^{(2d)}q+\cdot\cdot\cdot,\nonumber
\end{align}
where
\begin{align}
A_{3}&=2\pi\sqrt{-1}\sum_{\rho=0}^{l}(-1)^{\rho}\rho\wedge^{\rho}(W^{*})\otimes[-W^{*}-W+T^{*}+T-2(d-l)]\\
&+\sum_{\rho=0}^{l}(-1)^{\rho}\wedge^{\rho}(W^{*})\otimes[-l\pi\sqrt{-1}(-W^{*}-W+T^{*}+T-2(d-l))\nonumber\\
&-2\pi\sqrt{-1}(W^{*}-W)];\nonumber
\end{align}
\begin{align}
&\widetilde{a_{2}}(M, W, \tau)=\exp{\Big(\frac{1}{24}\cdot\frac{1}{30}(c_{2}(\mathcal{W}_{\alpha})+c_{2}(\mathcal{W}_{\beta}))\Big)}\Big\{-2\pi^{2}Td(M)ch\Big(\sum_{\rho=0}^{l}(-1)^{\rho}\\
&\cdot\Big(\rho-\frac{l}{2}\Big)^{2}\wedge^{\rho}(W^{*})\Big)\Big\}^{(2d)}+\exp{\Big(\frac{1}{24}\cdot\frac{1}{30}(c_{2}(\mathcal{W}_{\alpha})+c_{2}(\mathcal{W}_{\beta}))\Big)}\Big\{\frac{l}{6}\pi^{2}Td(M)\nonumber\\
&\cdot ch\Big(\sum_{\rho=0}^{l}(-1)^{\rho}\wedge^{\rho}(W^{*})\Big)\Big\}^{(2d)}+\cdot\cdot\cdot;\nonumber
\end{align}
\begin{align}
&\widetilde{a_{3}}(M, W, \tau)=\exp{\Big(\frac{1}{24}\cdot\frac{1}{30}(c_{2}(\mathcal{W}_{\alpha})+c_{2}(\mathcal{W}_{\beta}))\Big)}\Big\{-\frac{4}{3}\pi^{3}\sqrt{-1}Td(M)ch\Big(\sum_{\rho=0}^{l}(-1)^{\rho}\end{align}
\begin{align}
&\cdot\Big(\rho-\frac{l}{2}\Big)^{3}\wedge^{\rho}(W^{*})\Big)\Big\}^{(2d)}+\exp{\Big(\frac{1}{24}\cdot\frac{1}{30}(c_{2}(\mathcal{W}_{\alpha})+c_{2}(\mathcal{W}_{\beta}))\Big)}\Big\{\frac{l}{3}\pi^{3}\sqrt{-1}Td(M)\nonumber\\
&\cdot ch\Big(\sum_{\rho=0}^{l}(-1)^{\rho}\Big(\rho-\frac{l}{2}\Big)\wedge^{\rho}(W^{*})\Big)\Big\}^{(2d)}+\cdot\cdot\cdot;\nonumber
\end{align}
\begin{align}
&\widetilde{a_{4}}(M, W, \tau)=\exp{\Big(\frac{1}{24}\cdot\frac{1}{30}(c_{2}(\mathcal{W}_{\alpha})+c_{2}(\mathcal{W}_{\beta}))\Big)}\Big\{\frac{2}{3}\pi^{4}Td(M)ch\Big(\sum_{\rho=0}^{l}(-1)^{\rho}\\
&\cdot\Big(\rho-\frac{l}{2}\Big)^{4}\wedge^{\rho}(W^{*})\Big)\Big\}^{(2d)}+\exp{\Big(\frac{1}{24}\cdot\frac{1}{30}(c_{2}(\mathcal{W}_{\alpha})+c_{2}(\mathcal{W}_{\beta}))\Big)}\Big\{-\frac{l}{3}\pi^{4}Td(M)\nonumber\\
&\cdot ch\Big(\sum_{\rho=0}^{l}(-1)^{\rho}\Big(\rho-\frac{l}{2}\Big)^{2}\wedge^{\rho}(W^{*})\Big)\Big\}^{(2d)}+\exp{\Big(\frac{1}{24}\cdot\frac{1}{30}(c_{2}(\mathcal{W}_{\alpha})+c_{2}(\mathcal{W}_{\beta}))\Big)}\nonumber\\
&\cdot\Big\{\frac{l^{2}}{72}\pi^{4}Td(M)ch\Big(\sum_{\rho=0}^{l}(-1)^{\rho}\wedge^{\rho}(W^{*})\Big)\Big\}^{(2d)}+\cdot\cdot\cdot.\nonumber
\end{align}
\end{prop}

\begin{proof}
According to the condition $c_{1}(W)=c_{1}(M)=0$ and $p_{1}(W)=p_{1}(M),$ we calculate that
\begin{align}
\widetilde{a_{0}}(M, W, \tau)=&\Big\{\exp{\Big(\frac{1}{24}E_{2}(\tau)\cdot\frac{1}{30}(c_{2}(\mathcal{W}_{\alpha})+c_{2}(\mathcal{W}_{\beta}))\Big)}Td(M)ch(E(M, W, \tau, 0))\\
&\cdot\varphi(\tau)^{(16)}ch(\mathcal{V}_{\alpha})ch(\mathcal{V}_{\beta})\Big\}^{(2d)},\nonumber
\end{align}
where
\begin{align}
E(M, W, \tau, 0)&=\wedge_{-1}(W^{*})\{1+[-W^{*}-W+T^{*}+T-2(d-l)]q+[-W^{*}-W+\wedge^{2}(W^{*})\\
&+\wedge^{2}(W)+W^{*}\otimes W-W^{*}\otimes T^{*}-W^{*}\otimes T-W\otimes T^{*}-W\otimes T+T^{*}+T\nonumber\\
&+S^{2}(T^{*})+S^{2}(T)+T^{*}\otimes T+(d-l)(2d-2l-3)]q^{2}+\cdot\cdot\cdot\}.\nonumber
\end{align}
Write
\begin{align}
Ell(M, W, \mathcal{V}_{\alpha}, \mathcal{V}_{\beta}, \tau, z):=\widetilde{B_{0}}(z)+\widetilde{B_{1}}(z)q+\widetilde{B_{2}}(z)q^{2}+\cdot\cdot\cdot.
\end{align}
Consequently, we have
\begin{align}
&\widetilde{B_{0}}(z)=\exp{\Big(\frac{1}{24}\cdot\frac{1}{30}(c_{2}(\mathcal{W}_{\alpha})+c_{2}(\mathcal{W}_{\beta}))\Big)}\Big\{Td(M)ch\Big(\sum_{\rho=0}^{l}(-1)^{\rho}\wedge^{\rho}(W^{*})\Big)\Big\}^{(2d)}\\
&+\exp{\Big(\frac{1}{24}\cdot\frac{1}{30}(c_{2}(\mathcal{W}_{\alpha})+c_{2}(\mathcal{W}_{\beta}))\Big)}\Big\{2\pi\sqrt{-1}Td(M)ch\Big(\sum_{\rho=0}^{l}(-1)^{\rho}\Big(\rho-\frac{l}{2}\Big)\wedge^{\rho}(W^{*})\Big)\Big\}^{(2d)}z\nonumber
\end{align}
\begin{align}
&+\exp{\Big(\frac{1}{24}\cdot\frac{1}{30}(c_{2}(\mathcal{W}_{\alpha})+c_{2}(\mathcal{W}_{\beta}))\Big)}\Big\{-2\pi^{2}Td(M)ch\Big(\sum_{\rho=0}^{l}(-1)^{\rho}\Big(\rho-\frac{l}{2}\Big)^{2}\wedge^{\rho}(W^{*})\Big)\Big\}^{(2d)}z^{2}\nonumber\\
&+\exp{\Big(\frac{1}{24}\cdot\frac{1}{30}(c_{2}(\mathcal{W}_{\alpha})+c_{2}(\mathcal{W}_{\beta}))\Big)}\Big\{-\frac{4}{3}\pi^{3}\sqrt{-1}Td(M)ch\Big(\sum_{\rho=0}^{l}(-1)^{\rho}\Big(\rho-\frac{l}{2}\Big)^{3}\nonumber\\
&\cdot\wedge^{\rho}(W^{*})\Big)\Big\}^{(2d)}z^{3}+\exp{\Big(\frac{1}{24}\cdot\frac{1}{30}(c_{2}(\mathcal{W}_{\alpha})+c_{2}(\mathcal{W}_{\beta}))\Big)}\Big\{\frac{2}{3}\pi^{4}Td(M)ch\Big(\sum_{\rho=0}^{l}(-1)^{\rho}\Big(\rho-\frac{l}{2}\Big)^{4}\nonumber\\
&\cdot\wedge^{\rho}(W^{*})\Big)\Big\}^{(2d)}z^{4}+\cdot\cdot\cdot;\nonumber
\end{align}
\begin{align}
&\widetilde{B_{1}}(z)=\exp{\Big(\frac{1}{24}\cdot\frac{1}{30}(c_{2}(\mathcal{W}_{\alpha})+c_{2}(\mathcal{W}_{\beta}))\Big)}\Big\{Td(M)\Big(-16-\frac{1}{30}(c_{2}(\mathcal{W}_{\alpha})+c_{2}(\mathcal{W}_{\beta}))\\
&+ch(\mathcal{W}_{\alpha})+ch(\mathcal{W}_{\beta})\Big)ch\Big(\sum_{\rho=0}^{l}(-1)^{\rho}\wedge^{\rho}(W^{*})\Big)+Td(M)ch\Big(\sum_{\rho=0}^{l}(-1)^{\rho}\wedge^{\rho}(W^{*})\nonumber\\
&\otimes[-W^{*}-W+T^{*}+T-2(d-l)]\Big)\Big\}^{(2d)}+\exp{\Big(\frac{1}{24}\cdot\frac{1}{30}(c_{2}(\mathcal{W}_{\alpha})+c_{2}(\mathcal{W}_{\beta}))\Big)}\nonumber\\
&\cdot\Big\{2\pi\sqrt{-1}Td(M)\Big(-16-\frac{1}{30}(c_{2}(\mathcal{W}_{\alpha})+c_{2}(\mathcal{W}_{\beta}))+ch(\mathcal{W}_{\alpha})+ch(\mathcal{W}_{\beta})\Big)ch\Big(\sum_{\rho=0}^{l}(-1)^{\rho}\nonumber\\
&\cdot\Big(\rho-\frac{l}{2}\Big)\wedge^{\rho}(W^{*})\Big)+Td(M)ch\Big(2\pi\sqrt{-1}\sum_{\rho=0}^{l}(-1)^{\rho}\rho\wedge^{\rho}(W^{*})\otimes[-W^{*}-W+T^{*}\nonumber\\
&+T-2(d-l)]+\sum_{\rho=0}^{l}(-1)^{\rho}\wedge^{\rho}(W^{*})\otimes[-l\pi\sqrt{-1}(-W^{*}-W+T^{*}+T-2(d-l))\nonumber\\
&-2\pi\sqrt{-1}(W^{*}-W)]\Big)\Big\}^{(2d)}z+\cdot\cdot\cdot.\nonumber
\end{align}
As in the proof of Proposition 3.2, we can have Proposition 4.1.
\end{proof}

\begin{thm}
Let $c_{1}(W)=c_{1}(M)=0$ and the first Pontrjagin classes $p_{1}(W)=p_{1}(M),$ then\\
1) if either $2d<16,$ $2d-l$ is odd or $2d<16, 2d-l\leq-6$ but $2d-l\neq-8,$ then
\begin{align}
&\{Td(M)ch(\wedge_{-1}(W^{*}))\}^{(2d)}=0;\\
&\Big\{Td(M)\Big(-16-\frac{1}{30}(c_{2}(\mathcal{W}_{\alpha})+c_{2}(\mathcal{W}_{\beta}))+ch(\mathcal{W}_{\alpha})+ch(\mathcal{W}_{\beta})\Big)ch(\wedge_{-1}(W^{*}))+Td(M)
\end{align}
\begin{align}
&\cdot ch(\wedge_{-1}(W^{*}))ch(A_{1})\Big\}^{(2d)}=0;\nonumber\\
&\Big\{Td(M)\Big(104+\frac{13}{30}(c_{2}(\mathcal{W}_{\alpha})+c_{2}(\mathcal{W}_{\beta}))+\frac{1}{1800}(c_{2}(\mathcal{W}_{\alpha})+c_{2}(\mathcal{W}_{\beta}))^{2}+ch(\overline{\mathcal{W}_{\alpha}})+ch(\overline{\mathcal{W}_{\beta}})\\
&-16ch(\mathcal{W}_{\alpha})-\frac{1}{30}(c_{2}(\mathcal{W}_{\alpha})+c_{2}(\mathcal{W}_{\beta}))ch(\mathcal{W}_{\alpha})-16ch(\mathcal{W}_{\beta})-\frac{1}{30}(c_{2}(\mathcal{W}_{\alpha})+c_{2}(\mathcal{W}_{\beta}))ch(\mathcal{W}_{\beta})\nonumber\\
&+ch(\mathcal{W}_{\alpha})ch(\mathcal{W}_{\beta})\Big)ch(\wedge_{-1}(W^{*}))+Td(M)\Big(-16-\frac{1}{30}(c_{2}(\mathcal{W}_{\alpha})+c_{2}(\mathcal{W}_{\beta}))+ch(\mathcal{W}_{\alpha})\nonumber\\
&+ch(\mathcal{W}_{\beta})\Big)ch(\wedge_{-1}(W^{*}))ch(A_{1})+Td(M)ch(\wedge_{-1}(W^{*}))ch(A_{2})\Big\}^{(2d)}=0,\nonumber
\end{align}
2) if either $2d<16,$ $2d-l$ is even or $2d<16, 2d-l\leq-7$ but $2d-l\neq-9,$ then
\begin{align}
&\Big\{2\pi\sqrt{-1}Td(M)ch\Big(\sum_{\rho=0}^{l}(-1)^{\rho}\Big(\rho-\frac{l}{2}\Big)\wedge^{\rho}(W^{*})\Big)\Big\}^{(2d)}=0;\\
&\Big\{2\pi\sqrt{-1}Td(M)\Big(-16-\frac{1}{30}(c_{2}(\mathcal{W}_{\alpha})+c_{2}(\mathcal{W}_{\beta}))+ch(\mathcal{W}_{\alpha})+ch(\mathcal{W}_{\beta})\Big)ch\Big(\sum_{\rho=0}^{l}(-1)^{\rho}\\
&\cdot\Big(\rho-\frac{l}{2}\Big)\wedge^{\rho}(W^{*})\Big)+Td(M)ch(A_{3})\Big\}^{(2d)}=0,\nonumber
\end{align}
3) if either $2d<16,$ $2d-l$ is odd or $2d<16, 2d-l\leq-8$ but $2d-l\neq-10,$ then
\begin{align}
&\Big\{-2\pi^{2}Td(M)ch\Big(\sum_{\rho=0}^{l}(-1)^{\rho}\Big(\rho-\frac{l}{2}\Big)^{2}\wedge^{\rho}(W^{*})\Big)\Big\}^{(2d)}+\Big\{\frac{l}{6}\pi^{2}Td(M)ch\Big(\sum_{\rho=0}^{l}(-1)^{\rho}\\
&\cdot\wedge^{\rho}(W^{*})\Big)\Big\}^{(2d)}=0,\nonumber
\end{align}
4) if either $2d<16,$ $2d-l$ is even or $2d<16, 2d-l\leq-9$ but $2d-l\neq-11,$ then
\begin{align}
&\Big\{-\frac{4}{3}\pi^{3}\sqrt{-1}Td(M)ch\Big(\sum_{\rho=0}^{l}(-1)^{\rho}\Big(\rho-\frac{l}{2}\Big)^{3}\wedge^{\rho}(W^{*})\Big)\Big\}^{(2d)}+\Big\{\frac{l}{3}\pi^{3}\sqrt{-1}Td(M)\\
&\cdot ch\Big(\sum_{\rho=0}^{l}(-1)^{\rho}\Big(\rho-\frac{l}{2}\Big)\wedge^{\rho}(W^{*})\Big)\Big\}^{(2d)}=0,\nonumber
\end{align}
5) if either $2d<16,$ $2d-l$ is odd or $2d<16, 2d-l\leq-10$ but $2d-l\neq-12,$ then
\begin{align}
&\Big\{\frac{2}{3}\pi^{4}Td(M)ch\Big(\sum_{\rho=0}^{l}(-1)^{\rho}\Big(\rho-\frac{l}{2}\Big)^{4}\wedge^{\rho}(W^{*})\Big)\Big\}^{(2d)}+\Big\{-\frac{l}{3}\pi^{4}Td(M)ch\Big(\sum_{\rho=0}^{l}(-1)^{\rho}\\
&\cdot\Big(\rho-\frac{l}{2}\Big)^{2}\wedge^{\rho}(W^{*})\Big)\Big\}^{(2d)}+\Big\{\frac{l^{2}}{72}\pi^{4}Td(M)ch\Big(\sum_{\rho=0}^{l}(-1)^{\rho}\wedge^{\rho}(W^{*})\Big)\Big\}^{(2d)}=0.\nonumber
\end{align}
\end{thm}

\begin{thm}
Let $c_{1}(W)=c_{1}(M)=0$ and the first Pontrjagin classes $p_{1}(W)=p_{1}(M),$ then\\
1) if $2d-l=-4,$ then
\begin{align}
&\Big\{Td(M)\Big(-16-\frac{1}{30}(c_{2}(\mathcal{W}_{\alpha})+c_{2}(\mathcal{W}_{\beta}))+ch(\mathcal{W}_{\alpha})+ch(\mathcal{W}_{\beta})\Big)ch(\wedge_{-1}(W^{*}))+Td(M)\\
&\cdot ch(\wedge_{-1}(W^{*}))ch(A_{1})\Big\}^{(2d)}=240\{Td(M)ch(\wedge_{-1}(W^{*}))\}^{(2d)};\nonumber\\
&\Big\{Td(M)\Big(104+\frac{13}{30}(c_{2}(\mathcal{W}_{\alpha})+c_{2}(\mathcal{W}_{\beta}))+\frac{1}{1800}(c_{2}(\mathcal{W}_{\alpha})+c_{2}(\mathcal{W}_{\beta}))^{2}+ch(\overline{\mathcal{W}_{\alpha}})+ch(\overline{\mathcal{W}_{\beta}})\\
&-16ch(\mathcal{W}_{\alpha})-\frac{1}{30}(c_{2}(\mathcal{W}_{\alpha})+c_{2}(\mathcal{W}_{\beta}))ch(\mathcal{W}_{\alpha})-16ch(\mathcal{W}_{\beta})-\frac{1}{30}(c_{2}(\mathcal{W}_{\alpha})+c_{2}(\mathcal{W}_{\beta}))ch(\mathcal{W}_{\beta})\nonumber\\
&+ch(\mathcal{W}_{\alpha})ch(\mathcal{W}_{\beta})\Big)ch(\wedge_{-1}(W^{*}))+Td(M)\Big(-16-\frac{1}{30}(c_{2}(\mathcal{W}_{\alpha})+c_{2}(\mathcal{W}_{\beta}))+ch(\mathcal{W}_{\alpha})\nonumber\\
&+ch(\mathcal{W}_{\beta})\Big)ch(\wedge_{-1}(W^{*}))ch(A_{1})+Td(M)ch(\wedge_{-1}(W^{*}))ch(A_{2})\Big\}^{(2d)}=2160\{Td(M)\nonumber\\
&\cdot ch(\wedge_{-1}(W^{*}))\}^{(2d)},\nonumber
\end{align}
where
\begin{align}
A_{1}&=-W^{*}-W+T^{*}+T+l+4;\\
A_{2}&=-W^{*}-W+\wedge^{2}(W^{*})+\wedge^{2}(W)+W^{*}\otimes W-W^{*}\otimes T^{*}\\
&-W^{*}\otimes T-W\otimes T^{*}-W\otimes T+T^{*}+T+S^{2}(T^{*})+S^{2}(T)\nonumber\\
&+T^{*}\otimes T-(d-l)(7+l),\nonumber
\end{align}
2) if $2d-l=-2,$ then
\begin{align}
&\Big\{Td(M)\Big(-16-\frac{1}{30}(c_{2}(\mathcal{W}_{\alpha})+c_{2}(\mathcal{W}_{\beta}))+ch(\mathcal{W}_{\alpha})+ch(\mathcal{W}_{\beta})\Big)ch(\wedge_{-1}(W^{*}))+Td(M)\\
&\cdot ch(\wedge_{-1}(W^{*}))ch(A_{1})\Big\}^{(2d)}=-504\{Td(M)ch(\wedge_{-1}(W^{*}))\}^{(2d)};\nonumber
\end{align}
\begin{align}
&\Big\{Td(M)\Big(104+\frac{13}{30}(c_{2}(\mathcal{W}_{\alpha})+c_{2}(\mathcal{W}_{\beta}))+\frac{1}{1800}(c_{2}(\mathcal{W}_{\alpha})+c_{2}(\mathcal{W}_{\beta}))^{2}+ch(\overline{\mathcal{W}_{\alpha}})+ch(\overline{\mathcal{W}_{\beta}})\\
&-16ch(\mathcal{W}_{\alpha})-\frac{1}{30}(c_{2}(\mathcal{W}_{\alpha})+c_{2}(\mathcal{W}_{\beta}))ch(\mathcal{W}_{\alpha})-16ch(\mathcal{W}_{\beta})-\frac{1}{30}(c_{2}(\mathcal{W}_{\alpha})+c_{2}(\mathcal{W}_{\beta}))ch(\mathcal{W}_{\beta})\nonumber\\
&+ch(\mathcal{W}_{\alpha})ch(\mathcal{W}_{\beta})\Big)ch(\wedge_{-1}(W^{*}))+Td(M)\Big(-16-\frac{1}{30}(c_{2}(\mathcal{W}_{\alpha})+c_{2}(\mathcal{W}_{\beta}))+ch(\mathcal{W}_{\alpha})\nonumber\\
&+ch(\mathcal{W}_{\beta})\Big)ch(\wedge_{-1}(W^{*}))ch(A_{1})+Td(M)ch(\wedge_{-1}(W^{*}))ch(A_{2})\Big\}^{(2d)}=-16632\{Td(M)\nonumber\\
&\cdot ch(\wedge_{-1}(W^{*}))\}^{(2d)},\nonumber
\end{align}
where
\begin{align}
A_{1}&=-W^{*}-W+T^{*}+T+l+2;\\
A_{2}&=-W^{*}-W+\wedge^{2}(W^{*})+\wedge^{2}(W)+W^{*}\otimes W-W^{*}\otimes T^{*}\\
&-W^{*}\otimes T-W\otimes T^{*}-W\otimes T+T^{*}+T+S^{2}(T^{*})+S^{2}(T)\nonumber\\
&+T^{*}\otimes T-(d-l)(5+l),\nonumber
\end{align}
3) if $2d-l=0,$ then
\begin{align}
&\Big\{Td(M)\Big(-16-\frac{1}{30}(c_{2}(\mathcal{W}_{\alpha})+c_{2}(\mathcal{W}_{\beta}))+ch(\mathcal{W}_{\alpha})+ch(\mathcal{W}_{\beta})\Big)ch(\wedge_{-1}(W^{*}))+Td(M)\\
&\cdot ch(\wedge_{-1}(W^{*}))ch(A_{1})\Big\}^{(2d)}=480\{Td(M)ch(\wedge_{-1}(W^{*}))\}^{(2d)};\nonumber\\
&\Big\{Td(M)\Big(104+\frac{13}{30}(c_{2}(\mathcal{W}_{\alpha})+c_{2}(\mathcal{W}_{\beta}))+\frac{1}{1800}(c_{2}(\mathcal{W}_{\alpha})+c_{2}(\mathcal{W}_{\beta}))^{2}+ch(\overline{\mathcal{W}_{\alpha}})+ch(\overline{\mathcal{W}_{\beta}})\\
&-16ch(\mathcal{W}_{\alpha})-\frac{1}{30}(c_{2}(\mathcal{W}_{\alpha})+c_{2}(\mathcal{W}_{\beta}))ch(\mathcal{W}_{\alpha})-16ch(\mathcal{W}_{\beta})-\frac{1}{30}(c_{2}(\mathcal{W}_{\alpha})+c_{2}(\mathcal{W}_{\beta}))ch(\mathcal{W}_{\beta})\nonumber\\
&+ch(\mathcal{W}_{\alpha})ch(\mathcal{W}_{\beta})\Big)ch(\wedge_{-1}(W^{*}))+Td(M)\Big(-16-\frac{1}{30}(c_{2}(\mathcal{W}_{\alpha})+c_{2}(\mathcal{W}_{\beta}))+ch(\mathcal{W}_{\alpha})\nonumber\\
&+ch(\mathcal{W}_{\beta})\Big)ch(\wedge_{-1}(W^{*}))ch(A_{1})+Td(M)ch(\wedge_{-1}(W^{*}))ch(A_{2})\Big\}^{(2d)}=61920\{Td(M)\nonumber\\
&\cdot ch(\wedge_{-1}(W^{*}))\}^{(2d)},\nonumber
\end{align}
where
\begin{align}
A_{1}&=-W^{*}-W+T^{*}+T+l;\\
A_{2}&=-W^{*}-W+\wedge^{2}(W^{*})+\wedge^{2}(W)+W^{*}\otimes W-W^{*}\otimes T^{*}\\
&-W^{*}\otimes T-W\otimes T^{*}-W\otimes T+T^{*}+T+S^{2}(T^{*})+S^{2}(T)\nonumber\\
&+T^{*}\otimes T-(d-l)(3+l),\nonumber
\end{align}
4) if $2d-l=2,$ then
\begin{align}
&\Big\{Td(M)\Big(-16-\frac{1}{30}(c_{2}(\mathcal{W}_{\alpha})+c_{2}(\mathcal{W}_{\beta}))+ch(\mathcal{W}_{\alpha})+ch(\mathcal{W}_{\beta})\Big)ch(\wedge_{-1}(W^{*}))+Td(M)\\
&\cdot ch(\wedge_{-1}(W^{*}))ch(A_{1})\Big\}^{(2d)}=-264\{Td(M)ch(\wedge_{-1}(W^{*}))\}^{(2d)};\nonumber\\
&\Big\{Td(M)\Big(104+\frac{13}{30}(c_{2}(\mathcal{W}_{\alpha})+c_{2}(\mathcal{W}_{\beta}))+\frac{1}{1800}(c_{2}(\mathcal{W}_{\alpha})+c_{2}(\mathcal{W}_{\beta}))^{2}+ch(\overline{\mathcal{W}_{\alpha}})+ch(\overline{\mathcal{W}_{\beta}})\\
&-16ch(\mathcal{W}_{\alpha})-\frac{1}{30}(c_{2}(\mathcal{W}_{\alpha})+c_{2}(\mathcal{W}_{\beta}))ch(\mathcal{W}_{\alpha})-16ch(\mathcal{W}_{\beta})-\frac{1}{30}(c_{2}(\mathcal{W}_{\alpha})+c_{2}(\mathcal{W}_{\beta}))ch(\mathcal{W}_{\beta})\nonumber\\
&+ch(\mathcal{W}_{\alpha})ch(\mathcal{W}_{\beta})\Big)ch(\wedge_{-1}(W^{*}))+Td(M)\Big(-16-\frac{1}{30}(c_{2}(\mathcal{W}_{\alpha})+c_{2}(\mathcal{W}_{\beta}))+ch(\mathcal{W}_{\alpha})\nonumber\\
&+ch(\mathcal{W}_{\beta})\Big)ch(\wedge_{-1}(W^{*}))ch(A_{1})+Td(M)ch(\wedge_{-1}(W^{*}))ch(A_{2})\Big\}^{(2d)}=-135432\{Td(M)\nonumber\\
&\cdot ch(\wedge_{-1}(W^{*}))\}^{(2d)},\nonumber
\end{align}
where
\begin{align}
A_{1}&=-W^{*}-W+T^{*}+T+l-2;\\
A_{2}&=-W^{*}-W+\wedge^{2}(W^{*})+\wedge^{2}(W)+W^{*}\otimes W-W^{*}\otimes T^{*}\\
&-W^{*}\otimes T-W\otimes T^{*}-W\otimes T+T^{*}+T+S^{2}(T^{*})+S^{2}(T)\nonumber\\
&+T^{*}\otimes T-(d-l)(1+l),\nonumber
\end{align}
5) if $2d-l=4,$ then
\begin{align}
&\Big\{Td(M)\Big(104+\frac{13}{30}(c_{2}(\mathcal{W}_{\alpha})+c_{2}(\mathcal{W}_{\beta}))+\frac{1}{1800}(c_{2}(\mathcal{W}_{\alpha})+c_{2}(\mathcal{W}_{\beta}))^{2}+ch(\overline{\mathcal{W}_{\alpha}})+ch(\overline{\mathcal{W}_{\beta}})\\
&-16ch(\mathcal{W}_{\alpha})-\frac{1}{30}(c_{2}(\mathcal{W}_{\alpha})+c_{2}(\mathcal{W}_{\beta}))ch(\mathcal{W}_{\alpha})-16ch(\mathcal{W}_{\beta})-\frac{1}{30}(c_{2}(\mathcal{W}_{\alpha})+c_{2}(\mathcal{W}_{\beta}))ch(\mathcal{W}_{\beta})\nonumber\\
&+ch(\mathcal{W}_{\alpha})ch(\mathcal{W}_{\beta})\Big)ch(\wedge_{-1}(W^{*}))+Td(M)\Big(-16-\frac{1}{30}(c_{2}(\mathcal{W}_{\alpha})+c_{2}(\mathcal{W}_{\beta}))+ch(\mathcal{W}_{\alpha})\nonumber\\
&+ch(\mathcal{W}_{\beta})\Big)ch(\wedge_{-1}(W^{*}))ch(A_{1})+Td(M)ch(\wedge_{-1}(W^{*}))ch(A_{2})\Big\}^{(2d)}=196560\{Td(M)\nonumber\\
&\cdot ch(\wedge_{-1}(W^{*}))\}^{(2d)}-24\Big\{Td(M)\Big(-16-\frac{1}{30}(c_{2}(\mathcal{W}_{\alpha})+c_{2}(\mathcal{W}_{\beta}))+ch(\mathcal{W}_{\alpha})+ch(\mathcal{W}_{\beta})\Big)\nonumber\\
&\cdot ch(\wedge_{-1}(W^{*}))+Td(M)ch(\wedge_{-1}(W^{*}))ch(A_{1})\Big\}^{(2d)},\nonumber
\end{align}
where
\begin{align}
A_{1}&=-W^{*}-W+T^{*}+T+l-4;\\
A_{2}&=-W^{*}-W+\wedge^{2}(W^{*})+\wedge^{2}(W)+W^{*}\otimes W-W^{*}\otimes T^{*}\\
&-W^{*}\otimes T-W\otimes T^{*}-W\otimes T+T^{*}+T+S^{2}(T^{*})+S^{2}(T)\nonumber\\
&+T^{*}\otimes T+(d-l)(1-l),\nonumber
\end{align}
6) if $2d-l=6,$ then
\begin{align}
&\Big\{Td(M)\Big(-16-\frac{1}{30}(c_{2}(\mathcal{W}_{\alpha})+c_{2}(\mathcal{W}_{\beta}))+ch(\mathcal{W}_{\alpha})+ch(\mathcal{W}_{\beta})\Big)ch(\wedge_{-1}(W^{*}))+Td(M)\\
&\cdot ch(\wedge_{-1}(W^{*}))ch(A_{1})\Big\}^{(2d)}=-24\{Td(M)ch(\wedge_{-1}(W^{*}))\}^{(2d)};\nonumber\\
&\Big\{Td(M)\Big(104+\frac{13}{30}(c_{2}(\mathcal{W}_{\alpha})+c_{2}(\mathcal{W}_{\beta}))+\frac{1}{1800}(c_{2}(\mathcal{W}_{\alpha})+c_{2}(\mathcal{W}_{\beta}))^{2}+ch(\overline{\mathcal{W}_{\alpha}})+ch(\overline{\mathcal{W}_{\beta}})\\
&-16ch(\mathcal{W}_{\alpha})-\frac{1}{30}(c_{2}(\mathcal{W}_{\alpha})+c_{2}(\mathcal{W}_{\beta}))ch(\mathcal{W}_{\alpha})-16ch(\mathcal{W}_{\beta})-\frac{1}{30}(c_{2}(\mathcal{W}_{\alpha})+c_{2}(\mathcal{W}_{\beta}))ch(\mathcal{W}_{\beta})\nonumber\\
&+ch(\mathcal{W}_{\alpha})ch(\mathcal{W}_{\beta})\Big)ch(\wedge_{-1}(W^{*}))+Td(M)\Big(-16-\frac{1}{30}(c_{2}(\mathcal{W}_{\alpha})+c_{2}(\mathcal{W}_{\beta}))+ch(\mathcal{W}_{\alpha})\nonumber\\
&+ch(\mathcal{W}_{\beta})\Big)ch(\wedge_{-1}(W^{*}))ch(A_{1})+Td(M)ch(\wedge_{-1}(W^{*}))ch(A_{2})\Big\}^{(2d)}=-196632\{Td(M)\nonumber\\
&\cdot ch(\wedge_{-1}(W^{*}))\}^{(2d)},\nonumber
\end{align}
where
\begin{align}
A_{1}&=-W^{*}-W+T^{*}+T+l-6;\\
A_{2}&=-W^{*}-W+\wedge^{2}(W^{*})+\wedge^{2}(W)+W^{*}\otimes W-W^{*}\otimes T^{*}\\
&-W^{*}\otimes T-W\otimes T^{*}-W\otimes T+T^{*}+T+S^{2}(T^{*})+S^{2}(T)\nonumber\\
&+T^{*}\otimes T+(d-l)(3-l),\nonumber
\end{align}
7) if $2d-l=8,$ then
\begin{align}
&\Big\{Td(M)\Big(104+\frac{13}{30}(c_{2}(\mathcal{W}_{\alpha})+c_{2}(\mathcal{W}_{\beta}))+\frac{1}{1800}(c_{2}(\mathcal{W}_{\alpha})+c_{2}(\mathcal{W}_{\beta}))^{2}+ch(\overline{\mathcal{W}_{\alpha}})+ch(\overline{\mathcal{W}_{\beta}})\\
&-16ch(\mathcal{W}_{\alpha})-\frac{1}{30}(c_{2}(\mathcal{W}_{\alpha})+c_{2}(\mathcal{W}_{\beta}))ch(\mathcal{W}_{\alpha})-16ch(\mathcal{W}_{\beta})-\frac{1}{30}(c_{2}(\mathcal{W}_{\alpha})+c_{2}(\mathcal{W}_{\beta}))ch(\mathcal{W}_{\beta})\nonumber
\end{align}
\begin{align}
&+ch(\mathcal{W}_{\alpha})ch(\mathcal{W}_{\beta})\Big)ch(\wedge_{-1}(W^{*}))+Td(M)\Big(-16-\frac{1}{30}(c_{2}(\mathcal{W}_{\alpha})+c_{2}(\mathcal{W}_{\beta}))+ch(\mathcal{W}_{\alpha})\nonumber\\
&+ch(\mathcal{W}_{\beta})\Big)ch(\wedge_{-1}(W^{*}))ch(A_{1})+Td(M)ch(\wedge_{-1}(W^{*}))ch(A_{2})\Big\}^{(2d)}=146880\{Td(M)\nonumber\\
&\cdot ch(\wedge_{-1}(W^{*}))\}^{(2d)}+216\Big\{Td(M)\Big(-16-\frac{1}{30}(c_{2}(\mathcal{W}_{\alpha})+c_{2}(\mathcal{W}_{\beta}))+ch(\mathcal{W}_{\alpha})+ch(\mathcal{W}_{\beta})\Big)\nonumber\\
&\cdot ch(\wedge_{-1}(W^{*}))+Td(M)ch(\wedge_{-1}(W^{*}))ch(A_{1})\Big\}^{(2d)},\nonumber
\end{align}
where
\begin{align}
A_{1}&=-W^{*}-W+T^{*}+T+l-8;\\
A_{2}&=-W^{*}-W+\wedge^{2}(W^{*})+\wedge^{2}(W)+W^{*}\otimes W-W^{*}\otimes T^{*}\\
&-W^{*}\otimes T-W\otimes T^{*}-W\otimes T+T^{*}+T+S^{2}(T^{*})+S^{2}(T)\nonumber\\
&+T^{*}\otimes T+(d-l)(5-l).\nonumber
\end{align}
\end{thm}

\section{ Acknowledgements }

The author was supported in part by National Natural Science Foundation of China (NSFC) No.11771070. 
The author thanks the referee for his (or her) careful reading and helpful comments.

\vskip 1 true cm


\bigskip
\bigskip

\noindent {\footnotesize {\it S. Liu} \\
{School of Mathematics and Statistics, Northeast Normal University, Changchun 130024, China}\\
{Email: liusy719@nenu.edu.cn}

\noindent {\footnotesize {\it Y. Wang} \\
{School of Mathematics and Statistics, Northeast Normal University, Changchun 130024, China}\\
{Email: wangy581@nenu.edu.cn}

\end{document}